\documentclass[11pt]{amsart}

\usepackage{amsmath}

\usepackage{amssymb}

\usepackage{amsthm}

\usepackage{color}

\newtheorem{thm}{Theorem}[section]

\newtheorem{prop}{Proposition}[section]

\newtheorem{cor}{Corollary}[section]

\newtheorem{rem}{Remark}[section]

\def\<{\left<}\def\>{\right>}
\def\({\left(}\def\){\right)}

\allowdisplaybreaks

\begin{document}

\title{ Draw-down Parisian ruin for spectrally negative L\'{e}vy process}
\author{Wenyuan Wang and  Xiaowen Zhou}

\address{Wenyuan Wang: School of Mathematical Sciences, Xiamen University, Fujian 361005, People's Republic of China;
Email address: wwywang@xmu.edu.cn.\\
Xiaowen Zhou: Department of Mathematics and Statistics, Concordia University, Montreal, Quebec, Canada H3G 1M8; Email address:  xiaowen.zhou@concordia.ca.}


\begin{abstract}
In this paper we study the draw-down related Parisian ruin problem  for spectrally negative L\'{e}vy risk processes. We first introduce the draw-down Parisian ruin time and solve the corresponding two-sided exit problems via excursion theory. We also  find an expression of the potential measure for the process killed at the draw-down Parisian time. As applications, we obtained new results  for spectrally negative L\'{e}vy risk process with dividend barrier and with Parisian ruin.
\end{abstract}

\maketitle

\section{Introduction }
\setcounter{section}{1}

The concept of Parisian stopping time was first proposed in Chesney and Jeanblanc (1997) for option pricing in mathematical finance. Dassios and Wu (2009a, 2009b) later introduced Parisian ruin time for  linear Brownian motion and the Cramer-Lundberg risk processes to model the ruin problem with implemental delay, where expressions for the Parisian ruin probability were provided. Intuitively, for the risk process the Parisian ruin time is the first time  until when the surplus process has stayed  below level $0$ continuously during  a time period with a pre-determined duration of length $r$.

The Parisian ruin problem has since been studied extensively under the framework of spectrally negative L\'evy processes. By considering the spectrally negative L\'evy processes of bounded and unbounded variation, Czarna and Palmowski (2011) found the respective expressions for the Parisian ruin probability.
Loeffen et al. (2013) re-visited the Parisian ruin probability and provided an expression which is considerably simpler than the one of Czarna and Palmowski (2011), and unifies the results for spectrally negative L\'{e}vy processes of bounded and unbounded variation. In Lkabous et al. (2017), the result of Loeffen et al. (2013) was further extended to the  refracted L\'{e}vy processes. The Parisian ruin related dividend optimization problem was investigated in Czarna and Palmowski (2014), where the barrier dividend strategy turned out to be the optimal strategy.
Works on a variant  of the above model in which the duration $r$ is random can be found in Landriault et al. (2014), Baurdoux et al. (2016) and Frostig et al. (2019). Recent work concerning the Parisian ruin with an ultimate bankruptcy level  can be found in Czarna (2016), Czarna and Renaud (2016) and  Cheung and Wong (2017).

 The results on Parisian ruin are often expressed using the scale functions and the marginal density for the spectrally negative L\'evy process. The approaches in the previous literature on the Parisian ruin for L\'evy risk processes typically involve  arguments of  fluctuation identities if the  underlying L\'evy process has sample paths of bounded variation.
  Approximation and limiting arguments are further needed to handle the case of unbounded variation.

More recently, in Loeffen et al. (2018) a novel approach is adopted by connecting the desired Parisian ruin fluctuation quantity with solution to the  Kolmogorov forward equation for spectrally negative L\'evy process to find  the joint Laplace
transform of the Parisian ruin time and the Parisian ruin position, and an expression of the $q$-potential measure of the process
killed at the Parisian ruin time.

Since the Parisian ruin is defined using excursions of the underlying process, one would expect the excursion theory to play a role in its investigation. But we are not aware of previous study of Parisian ruin problems via excursion theory.

A general draw-down time for a stochastic process is a downward first passage time that depends on the previous supreme of the process. It generalizes the classical ruin time and helps to understand the path dependent relative downward fluctuations from the previous supreme for the underlying process.

 The draw-down time was first studied for diffusions in Lehoczky (1977). Some early work on draw-down time for spectrally negative L\'evy processes can be found in Pistorius (2007). In Avram et al. (2017) the draw-down exit problems were studied for taxed spectrally negative L\'evy processes using  both excursion theory and an approximation approach. More recent fluctuation results concerning the draw-down times for spectrally negative L\'evy processes such as the associated joint distribution, the potential measure and the creeping behaviors were obtained in Li et al. (2017) via  excursion theory. Many ruin time related results for  spectrally negative L\'evy risk processes can be generalized to the associated draw-down time setting, and at the same time, the obtained expressions are in terms of scale functions that remain semi-explicit. We refer to Wang and Zhou (2018) for recent work on draw-down reflected spectrally negative L\'evy processes.

 Given the previous results on both the Parisian ruin probability and the draw-down time, it comes naturally   to introduce the draw-down feature to the Parisian ruin problem for spectrally negative L\'evy risk processes. In this way the Parisian ruins  can be associated to the historical high of the process, which allows to pose more elaborate  Parisian ruin problem  and  helps to better understand the fluctuation  behaviors for Parisian ruin.
 In this paper we are going to implement this idea and generalize the known results on Parisian ruin time  to those concerning the draw-down Parisian ruin time.


 More precisely, for spectrally negative L\'evy risk processes we  find solutions to the two-sided exit problems associated to the draw-down Parisian ruin times. We also find an expression of potential measure associated to the draw-down Parisian ruin time. We also obtain recursive expressions for moments of accumulated time discounted  increments of the running  supreme up to the draw-down Parisian ruin time. As  applications, we recover a previous result and obtain new  results on Parisian ruin for a spectrally negative L\'evy risk process with a constant dividend barrier.

  To show the main results, we adopt the excursion theory approach that we find very handy for the draw-down fluctuation arguments for spectrally negative L\'evy processes.   To this end, we first identify the associated  exit quantity  under the excursion measure for the excursion process of reflected spectrally negative L\'evy process from its running supreme.  Since the draw-down related quantities can be expressed using the excursion process,   the desired results then follow from compensation formulas. A similar approach can be found in Li et al. (2017). To our best knowledge, this paper represents the first attempt of applying the excursion theory in the study of Parisian ruin problems.

The rest of the paper is arranged as follows. After the introduction in Section 1, in Section 2 we briefly review the spectrally negative L\'evy process, the associated scale functions, the draw-down time and several draw-down fluctuation results. The excursion process of the spectrally negative L\'evy process reflected from its previous supreme together with results on the excursion measure related to the Parisian ruin time are introduced in Section 3. The main results and proofs are contained in Section 4. The main results are applied to the spectrally negative L\'evy process with Parisian ruin and dividend barrier is presented  in Section 5 to recover previous results and to obtain new results.

\section{Preliminaries on spectrally negative L\'{e}vy processes and Parisian ruin problem}
\setcounter{section}{2}

We first briefly introduce the spectrally negative L\'{e}vy processes, the associated scale functions and some fluctuation
identities.
Write $X\equiv\{X(t);t\geq0\}$, defined on a probability space with probability  laws $\{\mathbb{P}_{x};x\in(-\infty,\infty)\}$ and natural filtration $\{\mathcal{F}_{t};t\geq0\}$, for a spectrally negative L\'{e}vy process that is not a purely increasing linear drift or the negative of a subordinator. Denote its running supreme process by
\[\bar{X}(t):=\sup\limits_{0\leq s\leq t}X(s), \,\,\, t\geq0.\]

The Laplace exponent of $X$ is given by
\begin{eqnarray}
\psi(\theta):=\ln \mathbb{E}_{x}\left(\mathrm{e}^{\theta (X(1)-x)}\right)=\gamma\theta+\frac{1}{2}\sigma^{2}\theta^{2}-\int_{(0,\infty)}\left(1-\mathrm{e}^{-\theta x}-\theta x\mathbf{1}_{(0,1)}(x)\right)\nu(\mathrm{d}x),\nonumber
\end{eqnarray}
where   the L\'{e}vy measure $\nu$  satisfies  $\int_{(0,\infty)}\left(1\wedge x^{2}\right)\nu(\mathrm{d}x)<\infty$.
It is known that $\psi(\theta)$ is finite for  $\theta\in[0,\infty)$, and it is strictly convex and infinitely differentiable.
As in Bertoin (1996), the $q$-scale functions $\{W^{(q)};q\geq0\}$ of $X$ are defined as follows. For each $q\geq0$, $W^{(q)}:\,[0,\infty)\rightarrow[0,\infty)$ is the unique strictly increasing and continuous function with Laplace transform
\begin{eqnarray}
\int_{0}^{\infty}\mathrm{e}^{-\theta x}W^{(q)}(x)\mathrm{d}x=\frac{1}{\psi(\theta)-q},\quad \mbox{for }\theta>\Phi_{q},\nonumber
\end{eqnarray}
where $\Phi_{q}$ is the largest solution of the equation $\psi(\theta)=q$. Further define $W^{(q)}(x)=0 $ for $x<0$, and write $W$ for the $0$-scale function $W^{(0)}$. Note that $W^{(q)}(0)=0 $ if and only if the process $X$ has sample paths of unbounded variation.

For $p, p+q\geq 0$, $y>0$
and $x\in(-\infty,\infty)$, define two more scale functions as
$$W^{(p,q)}_{y}(x):=W^{(p)}(x)+q\int_{y}^{x}W^{(p+q)}(x-w)W^{(p)}(w)\mathrm{d}w,$$
and
$$Z^{(p)}(x):=1+p\int_{0}^{x}W^{(p)}(w)\mathrm{d}w.$$

For any $x\in\mathbb{R}$ and $\vartheta\geq0$, define an exponential change of measure  for the
spectrally negative L\'{e}vy process by
\begin{eqnarray}
\left.\frac{\mathrm{d}\mathbb{P}_{x}^{\vartheta}}{\mathrm{d}\mathbb{P}_{x}}
\right|_{\mathcal{F}_{t}}=\mathrm{e}^{\vartheta(X(t)-x)-\psi(\vartheta)t}.\nonumber
\end{eqnarray}
Furthermore, under the probability measures $\mathbb{P}_{x}^{\vartheta}$, process $X$ remains a spectrally negative L\'{e}vy
process. From now on we  denote by  $W_{\vartheta}^{(q)}$ and $W_{\vartheta}$  the $q$-scale function and $0$-scale function, respectively,  under the measure $\mathbb{P}_{x}^{\vartheta}$.

For the process $X$, define its first up-crossing time and down-crossing time of level $a\in(-\infty,\infty)$, respectively, by
\begin{eqnarray}
\tau^{+}_{a}:=\inf\{t\geq0: X(t)>a\}\,\,\, \text{and}\,\,\, \tau_{a}^{-}:=\inf\{t\geq0: X(t)<a\}.\nonumber
\end{eqnarray}
It can be found in Kyprianou (2014) that
\begin{eqnarray}\label{t.s.0}
\mathbb{E}_{x}\left(\mathrm{e}^{-q \tau_{a}^{+}}\mathbf{1}_{\{\tau_{a}^{+}<\tau_{0}^{-}\}}\right)
= \frac{W^{(q)}(x)}{W^{(q)}(a)}, \quad x\in(-\infty,a].
\end{eqnarray}
In addition, it follows from Zhou (2007) that
\begin{eqnarray}\label{t.s.limit}
\lim\limits_{x\rightarrow\infty}\frac{W^{(q)\prime}(x)}{W^{(q)}(x)}=\Phi_{q}\,\,\, \text{and} \,\,\, \lim\limits_{y\rightarrow\infty}\frac{W^{(q)}(x+y)}{W^{(q)}(y)}=\mathrm{e}^{\Phi_{q}x}.
\end{eqnarray}

A function $\xi: (-\infty,\infty)\rightarrow (-\infty,\infty)$ is called a draw-down function if  $\xi(x)< x$ for all the values of $x$ that are of concern.
Define  the $\xi$-draw-down time $\tau_{\xi}$ of $X$ as
$$\tau_{\xi}:=\inf\{t\geq0: X(t)<\xi(\bar{X}(t))\}$$
with the convention that $\inf\emptyset:=\infty$. We call $\xi(\bar{X}(\tau_\xi))$ the associated draw-down level.
By Li et al. (2017), we have for $\bar{\xi}(z):=z-\xi(z)$,
\begin{eqnarray}\label{dra.d.t.s.}
\mathbb{E}_{x}\left(\mathrm{e}^{-q\tau_{a}^{+}}\mathbf{1}_{\{\tau_{a}^{+}<\tau_{\xi}\}}\right)
=\exp\left\{-\int_{x}^{a}\frac{W^{(q)\prime}(\bar{\xi}\left(z\right))}
{W^{(q)}(\bar{\xi}\left(z\right))}\mathrm{d}z\right\},\quad x\in(-\infty,a].
\end{eqnarray}

For $r>0$ the Parisian ruin time is defined by
$$\kappa_{r}: =\inf\{t>r: t-g_{t}>r\} \,\,\text{with} \,\, \inf\emptyset:=\infty, $$
where
$$ g_{t}: =\sup\{0\leq s \leq t: X(s) \geq0\} \,\,\text{with} \,\, \sup\emptyset:=0.$$
Given the draw-down function $\xi$, we define the $\xi$-draw-down Parisian ruin time of $X$ as
$$\kappa_{r}^{\xi}: =\inf\{t>r: t-g_{t}^{\xi}>r\} \,\,\text{with} \,\, \inf\emptyset:=\infty,$$
where
$$g_{t}^{\xi}: =\sup\{0\leq s \leq t: X(s) \geq\xi(\bar{X}(s))\}\,\,\text{with} \,\, \sup\emptyset:=0.$$

From Loeffen et al. (2013) and Czarna and Palmowski (2014), we have
\begin{eqnarray}\label{par.ruin.p.}
\mathbb{P}_{x}\left(\kappa_{r}<\infty\right)
=1-\mathbb{E}\left(X(1)\right)
\frac{\int_{0}^{\infty}W(x+z)z\mathbb{P}\left(X(r)\in\mathrm{d}z\right)}
{\int_{0}^{\infty}z\mathbb{P}\left(X(r)\in\mathrm{d}z\right)},\quad x\in(-\infty,\infty),
\nonumber
\end{eqnarray}
and
\begin{eqnarray}\label{two.sid.ex.}
\mathbb{E}_{x}\left(\mathrm{e}^{-q \tau_{a}^{+}}\mathbf{1}_{\{\tau_{a}^{+}<\kappa_{r}\}}\right)
=\frac{\ell_{r}^{(q)}(x)}{\ell_{r}^{(q)}(a)},\quad x\in(-\infty, a],
\end{eqnarray}
where
\begin{eqnarray}\label{}
\ell_{r}^{(q)}(x)\hspace{-0.3cm}&:=&\hspace{-0.3cm}
\int_{0}^{\infty}\mathrm{e}^{-\Phi_{q}z+qr}W^{(q)}(x+z)
\frac{z}{r}\mathbb{P}^{\Phi_{q}}\left(X(r)\in\mathrm{d}z\right)
\nonumber\\
\hspace{-0.3cm}&=&\hspace{-0.3cm}
\int_{0}^{\infty} W^{(q)}(x+z)
\frac{z}{r}\mathbb{P}\left(X(r)\in\mathrm{d}z\right)
.\nonumber
\end{eqnarray}
 Write $\ell_{r}:=\ell_{r}^{(0)}$ for simplicity. Then
\begin{equation}\label{parisian}
\mathbb{P}_{x}\left(\kappa_{r}<\infty\right)
=1-\frac{\mathbb{E}\left(X(1)\right)}
{\int_{0}^{\infty}\frac{z}{r}\mathbb{P}\left(X(r)\in\mathrm{d}z\right)}\ell_{r}(x).\nonumber
\end{equation}

For $b\in(0,\infty)$, let
\begin{eqnarray}\label{Div.pro.}
D(t):=\left(\bar{X}(t)- b\right)\vee0,\quad t\geq0,
\end{eqnarray}
 denote the accumulated amount of dividends paid until time $t$ of the barrier strategy with barrier at level $b$.

In this paper, we are interested in  the following draw-down Parisian ruin time related fluctuation quantities.
\begin{itemize}
\item[(\textbf{i})]
The draw-down Parisian ruin time related two-side exit problem
$$\mathbb{E}_{x}\left(\mathrm{e}^{-q \tau_{a}^{+}}\mathbf{1}_{\{\tau_{a}^{+}<\kappa_{r}^{\xi}\wedge \tau_{\eta}\}}\right), \quad x\in (-\infty,\infty),\,\,a\in[x,\infty)
,$$
where $\eta$ is another draw-down function such that $\eta(z)<\xi(z)< z$ for all $z\leq a$.
\item[(\textbf{ii})]
The joint Laplace transform involving the draw-down Parisian ruin time, the position of $X$ at the draw-down Parisian ruin time and its running superemum until the draw-down Parisian ruin time
$$
\mathbb{E}_{x}\left(\mathrm{e}^{-q \left(\kappa_{r}^{\xi}-r\right)}\mathrm{e}^{\lambda X(\kappa_{r}^{\xi})-\psi\left(\lambda\right)r}
\varphi(\bar{X}(\kappa_{r}^{\xi}))\mathbf{1}_{\{\kappa_{r}^{\xi}<\tau_{a}^{+}\}}\right)
,\quad x\in(-\infty,a],\,a\in(-\infty,\infty),$$
where $\varphi:\,(-\infty,\infty)\rightarrow(-\infty,\infty)$ is  an arbitrary bounded measurable function.
\item[(\textbf{iii})]
The potential measure of $X$ involving the draw-down Parisian ruin time
$$
\int_{0}^{\infty}\mathrm{e}^{-q \left(t-r\right)}
\mathbb{E}_{x}\left(f(X(t),\bar{X}(t)); \,t<\kappa_{r}^{\xi}\wedge \tau_{a}^{+}\right)
\mathrm{d}t,\quad x\in(-\infty,a],\,a\in(-\infty,\infty),
$$
where $f$ is an arbitrary bounded bivariate function which is differentiable with respect to the first argument.
\item[(\textbf{iv})]
The $k$-th moment of the accumulated  time discounted running supreme up to the draw-down Parisian ruin time
$$V_{k}^{\xi}(x;b):=\mathbb{E}_{x}\left(\left[D_{b}\right]^{k}\right),\quad x\in(-\infty,\infty),\,b\in(0,\infty),$$
with $D_{b}:=D_{\xi,b}:=\int_{0-}^{\kappa^{\xi}_{r}}\mathrm{e}^{-q t}\,\mathrm{d}D(t)$.
Further  let
$V_{k}(x):=V_{k}^{\xi}(x;x),\,\,x\in(-\infty,\infty)$.
For $\xi(x)=(x-b)\vee 0 $ with $b>0$, $D_{b}$ can be interpreted as the accumulated  discounted   dividends paid according to the barrier dividend strategy with barrier at level $b$ until the draw-down Parisian ruin time.
\end{itemize}

We assume  the differentiability of $\ell_{r}^{(q)}$  whenever needed.
In fact, by \eqref{t.s.limit} and the definition of $\ell_{r}^{(q)}$, $\ell_{r}^{(q)}$ inherits the same differentiability from $W^{(q)}$.
It is known that,
when $X$ has sample paths of unbounded variation, or when $X$ has sample paths of bounded variation and the L\'{e}vy measure has no atoms, the scale function $W^{(q)}$ (hence, $\ell_{r}^{(q)}$) is continuously differentiable over $(0, \infty)$ (resp, $(-\infty,\infty)$).
Moreover, if $X$ has a nontrivial Gaussian component, then $W^{(q)}$ (and hence, $\ell_{r}^{(q)}$) is twice continuously differentiable over $(0, \infty)$ (resp, $(-\infty,\infty)$).
The interested readers are referred to Chan et al. (2011) and Kuznetsov et al. (2012) for more detailed discussions on the smoothness of scale functions.

\section{The excursion process and Parisian related quantities under the excursion measure }
\setcounter{section}{3}
\setcounter{equation}{0}

In this section, we  briefly recall basic  concepts in excursion theory for the reflected process $\{\bar{X}(t)-X(t);t\geq0\}$, and we refer to Bertoin (1996) for more details. We also obtain  new Parisian ruin related results on the excursion measure.

For $x\in(-\infty,\infty)$, the process $\{L(t):= \bar{X}(t)-x, t\geq0\}$ is  a local time at $0$ for
the Markov process $\{\bar{X}(t)-X(t);t\geq0\}$ under $\mathbb{P}_{x}$.
The corresponding inverse local time is defined as
$$L^{-1}(t):=\inf\{s\geq0: L(s)>t\}=\sup\{s\geq0: L(s)\leq t\}.$$
Further, let $L^{-1}(t-):=\lim\limits_{s\uparrow t}L^{-1}(s)$.
Define a Poisson point process $\{(t, \varepsilon_{t}); t\geq0\}$ as
\begin{eqnarray}
\label{ex.pro.ind.by.lo.ti.}
\varepsilon_{t}(s):=X(L^{-1}(t))-X(L^{-1}(t-)+s), \,\,s\in(0,L^{-1}(t)-L^{-1}(t-)],
\end{eqnarray}
whenever the lifetime of $\varepsilon_{t}$ is strictly positive, i.e. $L^{-1}(t)-L^{-1}(t-)>0$.
If  $L^{-1}(t)-L^{-1}(t-)=0 $, define $\varepsilon_{t}:=\Upsilon$ with $\Upsilon$ being an additional isolated point.
It is known  that $\varepsilon$ is a Poisson point process with
characteristic measure $n$
if $\{\bar{X}(t)-X(t);t\geq0\}$ is recurrent; otherwise, $\{\varepsilon_{t}; t\leq L(\infty)\}$ is a Poisson point process stopped at the first excursion of infinite lifetime. Here, $n$ is a $\sigma$-finite measure on the space  $\mathcal{E}$ of excursions,
i.e. the space $\mathcal{E}$  of c\`{a}dl\`{a}g functions $f$ satisfying
$$f:\,(0,\zeta)\rightarrow (0,\infty)\,\,\text{ for some} \,\, \zeta\in(0,\infty] \,\, \text{ and} \,\, f(\zeta)\in (0, \infty) \,\, \text{if}\,\, \zeta<\infty,$$
where $\zeta\equiv\zeta(f)$ denotes the excursion length or lifetime; see Definition 6.13 of Kyprianou (2006) for the definition of $\mathcal{E}$.
Denote by $\varepsilon(\cdot)$, or $\varepsilon$ for short, a generic excursion
belonging to the space $\mathcal{E}$ of canonical excursions.
The excursion height of a canonical excursion $\varepsilon$ is
denoted by $\bar{\varepsilon}=\sup\limits_{t\in[0,\zeta]}\varepsilon(t)$.

With a little abuse of notation, for $a\in(0,\infty)$ and $t\in[0,\zeta]$, let
$$
g_{t}^{a}(\varepsilon):=
\left\{\begin{array}{ll} \inf\{s\in[0,\zeta]: s\leq t, \varepsilon(w)\geq a \mbox{ for all }w\in[s,t]\}\quad &\mbox{if } \,\varepsilon(t)\geq a,\\
t & \mbox{otherwise}.
\end{array}\right.$$
and
$$
d_{t}^{a}(\varepsilon):=
\left\{\begin{array}{ll}\sup\{s\in[0,\zeta]: s\geq t, \varepsilon(w)\geq a \mbox{ for all }w\in[t,s)\}\quad &\mbox{if } \,\varepsilon(t)\geq a,\\
t &\mbox{otherwise}.
\end{array}\right.$$

 Write $\zeta_{t}^{a}(\varepsilon):=d_{t}^{a}(\varepsilon)-g_{t}^{a}(\varepsilon)$ for the length of the maximum time interval (containing $t$) when the canonical excursion $\varepsilon$ stays above the level $a$.
Further define
$$\alpha_{a}^{+}(\varepsilon):=\inf\{g_{t}^{a}(\varepsilon): t\in[0,\zeta], \zeta_{t}^{a}(\varepsilon)>r\},$$
with the convention that $\inf\emptyset:=\zeta$. Intuitively, $\alpha_{a}^{+}(\varepsilon)$ is the starting time of the first time interval of length more than $r$ when the excursion path stays continuous above level $a$.

Denote by $\varepsilon_{g}$
the excursion (away from $0$)  with left-end point $g$ for the reflected process $\{\bar{X}(t)-X(t);t\geq0\}$, and $\zeta_{g}$ and $\bar{\varepsilon}_{g}$ denote its lifetime and excursion height, respectively; see Section IV.4 of Bertoin (1996).

\vspace{0.3cm}
The following result gives the excursion measure of the event that there exists a time interval with length at least $r$ during which either the excursion process continuously stays above level $z>0$, or there is an excursion with height strictly greater than $z+y$ for some $y>0$.

\begin{prop}\label{prop3.1}
For any $z,y\in(0,\infty)$, we have
\begin{eqnarray}\label{n.pari.dra.dow.pro.c.0}
n\left(\alpha_{z}^{+}(\varepsilon)<\zeta \mbox{ \emph{or} } \overline{\varepsilon}> z+y\right)
=\frac{W^{\prime}(z)\phi(y,r)+\chi^{\prime}(z,y,r)}{W(z)\phi(y,r)+\chi(z,y,r)},
\end{eqnarray}
where the derivative  of $\chi$ is with respect to the first argument, and
the Laplace transforms of $\phi(y,r)$ and $\chi(x,y,r)$ \emph{(}in $r$\emph{)} are given, respectively, by
\begin{eqnarray}\label{lap.1}
\,\,\int_{0}^{\infty}\mathrm{e}^{-\theta r}\chi(x,y,r)\,\mathrm{d}r
=\frac{1}{\theta}\left(\frac{W_{y}^{(\theta,-\theta)}(x+y)}{W^{(\theta)}(y)}
-\frac{W(x)
Z^{(\theta)}(y)}{W^{(\theta)}(y)}\right)
,\nonumber
\end{eqnarray}
and
\begin{eqnarray}\label{lap.2}
\int_{0}^{\infty}\mathrm{e}^{-\theta r}\phi(y,r)\,\mathrm{d}r
=
\frac{
Z^{(\theta)}(y)}
{\theta W^{(\theta)}(y)}
.\nonumber
\end{eqnarray}
\end{prop}

\begin{proof}
It follows from Theorem 1 of Czarna and Renaud (2016) that
\begin{eqnarray}
\label{pro.c.}
1-\mathbb{P}_{x}\left(\kappa_{r}\wedge \tau_{-y}^{-}<\infty\right)
=\mathbb{E}\left(X(1)\right)\left(W(x)+\frac{\chi(x,y,r)}{\phi(y,r)}\right),
\end{eqnarray}
for any fixed positive $y$.
By the strong Markov property we have for $a>x$,
\begin{eqnarray}
\label{}
\mathbb{P}_{x}\left(\kappa_{r}\wedge \tau_{-y}^{-}<\infty\right)
\hspace{-0.3cm}&=&\hspace{-0.3cm}
\mathbb{P}_{x}\left(\tau_{a}^{+}<\kappa_{r}\wedge \tau_{-y}^{-}<\infty\right)
+\mathbb{P}_{x}\left(\kappa_{r}\wedge \tau_{-y}^{-}<\tau_{a}^{+}\right)
\nonumber\\
\hspace{-0.3cm}&=&\hspace{-0.3cm}
\mathbb{P}_{x}\left(\tau_{a}^{+}<\kappa_{r}\wedge \tau_{-y}^{-}\right)
\mathbb{P}_{a}\left(\kappa_{r}\wedge \tau_{-y}^{-}<\infty\right)
+1-\mathbb{P}_{x}\left(\tau_{a}^{+}<\kappa_{r}\wedge \tau_{-y}^{-}\right),
\nonumber
\end{eqnarray}
which together with \eqref{pro.c.} implies
\begin{eqnarray}
\label{pro.xi.c.}
\mathbb{P}_{x}\left(\tau_{a}^{+}<\kappa_{r}\wedge \tau_{-y}^{-}\right)
\hspace{-0.3cm}&=&\hspace{-0.3cm}
\frac{1-\mathbb{P}_{x}\left(\kappa_{r}\wedge \tau_{-y}^{-}<\infty\right)}
{1-\mathbb{P}_{a}\left(\kappa_{r}\wedge \tau_{-y}^{-}<\infty\right)}
\nonumber\\
\hspace{-0.3cm}&=&\hspace{-0.3cm}\frac{W(x)\phi(y,r)+\chi(x,y,r)}
{W(a)\phi(y,r)+\chi(a,y,r)}.
\end{eqnarray}
Because $\{(t, \varepsilon_{t}); t\geq0\}$ defined via \eqref{ex.pro.ind.by.lo.ti.} is a Poisson point process with intensity measure $\mathrm{d}t\times\mathrm{d}n$,
we have
\begin{eqnarray}\label{pari.dra.dow.pro.c.0}
\mathbb{P}_{x}\left(\tau_{a}^{+}<\kappa_{r}\wedge \tau_{-y}^{-}\right)
\hspace{-0.3cm}&=&\hspace{-0.3cm}
\mathbb{E}_{x}\left(
\prod\limits_{t\leq a-x}\mathbf{1}_{\{\alpha_{x+t}^{+}(\varepsilon_{t})= \zeta(\varepsilon_{t}),\,\overline{\varepsilon}_{t}\leq x+t+y\}}\right)
\nonumber\\
\hspace{-0.3cm}&=&\hspace{-0.3cm}
\exp\left(-\int_{0}^{a-x}n\left(\alpha_{x+t}^{+}(\varepsilon)<\zeta \mbox{ or } \overline{\varepsilon}> x+t+y\right)\,\mathrm{d}t\right)
\nonumber\\
\hspace{-0.3cm}&=&\hspace{-0.3cm}
\exp\left(-\int_{x}^{a} n\left(\alpha_{w}^{+}(\varepsilon)<\zeta \mbox{ or } \overline{\varepsilon}> w+y\right)\,\mathrm{d}w\right),
\end{eqnarray}
where $\overline{\varepsilon}_{t}$ denotes the excursion height of $\varepsilon_{t}$. Combining \eqref{pro.xi.c.} and \eqref{pari.dra.dow.pro.c.0} yields \eqref{n.pari.dra.dow.pro.c.0}.
\end{proof}

\vspace{0.3cm}
We next prove  a version of Proposition \ref{prop3.1} for  $y=\infty$.

\begin{cor}\label{rem3.1.n}
For any $x\in(0,\infty)$, we have
\[ n\left(\alpha_{x}^{+}(\varepsilon)<\zeta\right)=\frac{\ell_{r}^{\prime}(x)}{\ell_{r}(x)}.  \]
\end{cor}

\begin{proof}
By definition we have
\begin{eqnarray}\label{lap.1.new}
\,\,\int_{0}^{\infty}\mathrm{e}^{-\theta r}\left(W(x)\phi(y,r)+\chi(x,y,r)\right)\mathrm{d}r
=\frac{W_{y}^{(\theta,-\theta)}(x+y)}{\theta W^{(\theta)}(y)}
.
\end{eqnarray}
The definition of $W^{(\theta,-\theta)}_{y}(x+y)$ together with \eqref{t.s.limit} yields
\begin{eqnarray}
\label{rem3.1.Lap.1.new}
\lim\limits_{y\uparrow\infty}\frac{W^{(\theta,-\theta)}_{y}(x+y)}{\theta W^{(\theta)}(y)}
\hspace{-0.3cm}&=&\hspace{-0.3cm}
\mathrm{e}^{\Phi_{\theta}x}
\left(\frac{1}{\theta}-\int_{0}^{x}W(w)\,
\mathrm{e}^{-\Phi_{\theta}w}\,\mathrm{d}w\right)
=
\int_{0}^{\infty}W(x+w)\,\mathrm{e}^{-\Phi_{\theta} w}\,\mathrm{d}w
,
\end{eqnarray}
which coincides with the Laplace transform (in $r$) of $\ell_{r}(x)$ as follows.
\begin{eqnarray}\label{la.id.02}
\int_{0}^{\infty}\mathrm{e}^{-\theta r}\ell_{r}(x)\,\mathrm{d}r
\hspace{-0.3cm}&=&\hspace{-0.3cm}
\int_{0}^{\infty}\mathrm{e}^{-\theta r}\int_{0}^{\infty}W(x+z)
\frac{z}{r}\mathbb{P}\left(X(r)\in\mathrm{d}z\right)\,\mathrm{d}r
\nonumber\\
\hspace{-0.3cm}&=&\hspace{-0.3cm}
\int_{0}^{\infty}\mathrm{e}^{-\theta r}\int_{0}^{\infty}W(x+z)
\mathbb{P}(\tau_{z}^{+}\in\mathrm{d}r)\mathrm{d}z
\nonumber\\
\hspace{-0.3cm}&=&\hspace{-0.3cm}
\int_{0}^{\infty}W(x+z)
\mathbb{E}(\mathrm{e}^{-\theta \tau_{z}^{+}})\mathrm{d}z
\nonumber\\
\hspace{-0.3cm}&=&\hspace{-0.3cm}
\int_{0}^{\infty}W(x+w)\,\mathrm{e}^{-\Phi_{\theta} w}\,\mathrm{d}w
,
\end{eqnarray}
where we have used the Kendall's identity
\[\frac{z}{r}\mathbb{P}\left(X(r)\in\mathrm{d}z\right)\,\mathrm{d}r
=\mathbb{P}(\tau_{z}^{+}\in\mathrm{d}r)\mathrm{d}z,\,z,r\geq0.\]
By \eqref{lap.1.new}, \eqref{rem3.1.Lap.1.new}, \eqref{la.id.02} and the continuity of Laplace transforms, we have
$$\lim\limits_{y\uparrow\infty}\left(W(x)\phi(y,r)+\chi(x,y,r)\right)=\ell_{r}(x).
$$
In fact, the same arguments as above lead to
$$\lim\limits_{y\uparrow\infty}\left(W^{\prime}(x)\phi(y,r)+\chi^{\prime}(x,y,r)\right)=\ell_{r}^{\prime}(x).$$
Therefore,
\begin{eqnarray}
n\left(\alpha_{x}^{+}(\varepsilon)<\zeta\right)
\hspace{-0.3cm}&=&\hspace{-0.3cm}
\lim\limits_{y\uparrow\infty}n\left(\alpha_{x}^{+}(\varepsilon)<\zeta \mbox{ \emph{or} } \overline{\varepsilon}> x+y\right)
\nonumber\\
\hspace{-0.3cm}&=&\hspace{-0.3cm}
\lim\limits_{y\uparrow\infty}\frac{W^{\prime}(x)\phi(y,r)+\chi^{\prime}(x,y,r)}{W(x)\phi(y,r)+\chi(x,y,r)}
=\frac{\ell_{r}^{\prime}(x)}{\ell_{r}(x)},\nonumber
\end{eqnarray}
which is the desired result.
\end{proof}

\vspace{0.3cm}
The following result gives the joint Laplace transform involving $\alpha_{a}^{+}$ under the excursion measure.

\begin{prop}\label{lemma2}
For any $q,\lambda\in[0,\infty)$ and $a,r\in(0,\infty)$, we have
\begin{eqnarray}\label{n.2}
\hspace{-0.3cm}&&\hspace{-0.3cm}
n\left(\mathrm{e}^{-q \alpha_{a}^{+}(\varepsilon)}\mathrm{e}^{\lambda \left(a-\varepsilon\left(\alpha_{a}^{+}(\varepsilon)+r\right)\right)-
\psi\left(\lambda\right)r}\mathbf{1}_{\{\alpha_{a}^{+}(\varepsilon)< \zeta\}}\right)
\nonumber\\
\hspace{-0.3cm}&=&\hspace{-0.3cm}
\frac{\ell_{r}^{(q)\prime}(a)}{\ell_{r}^{(q)}(a)}
\left(\mathrm{e}^{\lambda a}-\left(\psi(\lambda)-q\right)
\left(\mathrm{e}^{\lambda a}\int_{0}^{a}W^{(q)}(z)\mathrm{e}^{-\lambda z}\mathrm{d}z+\int_{0}^{r}
\mathrm{e}^{-\psi(\lambda) s}\,\ell_{s}^{(q)}(a)\mathrm{d}s\right)\right)
\nonumber\\
\hspace{-0.3cm}&&\hspace{-0.3cm}
-
\lambda \mathrm{e}^{\lambda a}+\left(\psi(\lambda)-q\right)
\left(\lambda \mathrm{e}^{\lambda a}\int_{0}^{a}W^{(q)}(z)\mathrm{e}^{-\lambda z}\mathrm{d}z+W^{(q)}(a)+\int_{0}^{r}
\mathrm{e}^{-\psi(\lambda) s}\,\ell_{s}^{(q)\prime}(a)\mathrm{d}s\right)
.
\end{eqnarray}
\end{prop}

\begin{proof}[Proof:]\,\,\,
Given $q,\lambda\geq0$, $r, a>0$ and $a\geq x$, by Theorem 3.1 in Loeffen et al. (2018) we have
\begin{eqnarray}\label{fluc.1}
\hspace{-0.3cm}&&\hspace{-0.3cm}
\mathbb{E}_{x}\left(\mathrm{e}^{-q \left(\kappa_{r}-r\right)}\mathrm{e}^{\lambda X(\kappa_{r})-\psi\left(\lambda\right)r}\mathbf{1}_{\{\kappa_{r}<\tau_{a}^{+}\}}\right)
\nonumber\\
\hspace{-0.3cm}&=&\hspace{-0.3cm}
\mathrm{e}^{\lambda x}-\left(\psi(\lambda)-q\right)
\left(\mathrm{e}^{\lambda x}\int_{0}^{x}W^{(q)}(z)\mathrm{e}^{-\lambda z}\mathrm{d}z+\int_{0}^{r}
\mathrm{e}^{-\psi(\lambda) s}\,\ell_{s}^{(q)}(x)\mathrm{d}s\right)
\nonumber\\
\hspace{-0.3cm}&&\hspace{-0.3cm}
-\frac{\ell_{r}^{(q)}(x)}{\ell_{r}^{(q)}(a)}
\left(\mathrm{e}^{\lambda a}-\left(\psi(\lambda)-q\right)
\left(\mathrm{e}^{\lambda a}\int_{0}^{a}W^{(q)}(z)\mathrm{e}^{-\lambda z}\mathrm{d}z+\int_{0}^{r}
\mathrm{e}^{-\psi(\lambda) s}\,\ell_{s}^{(q)}(a)\mathrm{d}s\right)\right).
\end{eqnarray}
By \eqref{two.sid.ex.} and the compensation formula, see for example Corollary 4.11  of Bertoin (1996) or Theorem 4.4  of Kyprianou (2014)), one gets
\begin{eqnarray}\label{}
\hspace{-0.3cm}&&\hspace{-0.3cm}
\mathbb{E}_{x}\left(\mathrm{e}^{-q \left(\kappa_{r}-r\right)}\mathrm{e}^{\lambda X(\kappa_{r})-\psi\left(\lambda\right)r}\mathbf{1}_{\{\kappa_{r}<\tau_{a}^{+}\}}\right)
\nonumber\\
\hspace{-0.3cm}&=&\hspace{-0.3cm}
\mathbb{E}_{x}\left(\sum_{g}\mathrm{e}^{-q g}\prod\limits_{h<g}\mathbf{1}_{\{\alpha_{x+L(h)}^{+}(\varepsilon_{h})= \,\zeta_{h},\,x+L(g)\leq a\}}
\mathrm{e}^{-q \alpha_{x+L(g)}^{+}(\varepsilon_{g})}\right.
\nonumber\\
\hspace{-0.3cm}&&\hspace{0.5cm}
\left.
\times\mathrm{e}^{\lambda \left(x+L(g)-\varepsilon_{g}\left(\alpha_{x+L(g)}^{+}(\varepsilon_{g})+r\right)\right)-\psi\left(\lambda\right)r}\mathbf{1}_{\{\alpha_{x+L(g)}^{+}(\varepsilon_{g})< \zeta_{g}\}}\right)
\nonumber\\
\hspace{-0.3cm}&=&\hspace{-0.3cm}
\mathbb{E}_{x}\left(\int_{0}^{\infty}\mathrm{e}^{-q w}\prod\limits_{h<w}\mathbf{1}_{\{\alpha_{x+L(h)}^{+}(\varepsilon_{h})= \,\zeta_{h},\,x+L(w)\leq a\}}
\int_{\mathcal{E}}\mathrm{e}^{-q \alpha_{x+L(w)}^{+}(\varepsilon)}\right.
\nonumber\\
\hspace{-0.3cm}&&\hspace{0.5cm}
\left.
\times\mathrm{e}^{\lambda \left(x+L(w)-\varepsilon\left(\alpha_{x+L(w)}^{+}(\varepsilon)+r\right)\right)-\psi\left(\lambda\right)r}\mathbf{1}_{\{\alpha_{x+L(w)}^{+}(\varepsilon)< \zeta\}}\,n(\,\mathrm{d}\varepsilon)\,\mathrm{d}L(w)\right)
\nonumber\\
\hspace{-0.3cm}&=&\hspace{-0.3cm}
\mathbb{E}_{x}\left(\int_{0}^{a-x}\mathrm{e}^{-q L^{-1}(w-)}
\prod\limits_{h<L^{-1}(w-)}\mathbf{1}_{\{\alpha_{x+L(h)}^{+}(\varepsilon_{h})= \,\zeta_{h}\}}
\right.
\nonumber\\
\hspace{-0.3cm}&&\hspace{0.5cm}
\left.
\times\int_{\mathcal{E}}\mathrm{e}^{-q \alpha_{x+w}^{+}(\varepsilon)}\mathrm{e}^{\lambda \left(x+w-\varepsilon\left(\alpha_{x+w}^{+}(\varepsilon)+r\right)\right)-\psi\left(\lambda\right)r}\mathbf{1}_{\{\alpha_{x+w}^{+}(\varepsilon)< \zeta\}}\,n(\,\mathrm{d}\varepsilon)\,\mathrm{d}w\right)
\nonumber\\
\hspace{-0.3cm}&=&\hspace{-0.3cm}
\int_{x}^{a}\mathbb{E}_{x}\left(\mathrm{e}^{-q \tau_{w}^{+}}\mathbf{1}_{\{\tau_{w}^{+}<\kappa_{r}\}}\right)
\,n\left(\mathrm{e}^{-q \alpha_{w}^{+}(\varepsilon)}\mathrm{e}^{\lambda \left(w-\varepsilon\left(\alpha_{w}^{+}(\varepsilon)+r\right)\right)-\psi\left(\lambda\right)r}\mathbf{1}_{\{\alpha_{w}^{+}(\varepsilon)< \zeta\}}\right)\,\mathrm{d}w
\nonumber\\
\hspace{-0.3cm}&=&\hspace{-0.3cm}
\int_{x}^{a}
\frac{\ell_{r}^{(q)}(x)}{\ell_{r}^{(q)}(w)}
\,n\left(\mathrm{e}^{-q \alpha_{w}^{+}(\varepsilon)}\mathrm{e}^{\lambda \left(w-\varepsilon\left(\alpha_{w}^{+}(\varepsilon)+r\right)\right)-\psi\left(\lambda\right)r}\mathbf{1}_{\{\alpha_{w}^{+}(\varepsilon)< \zeta\}}\right)\,\mathrm{d}w,\nonumber
\end{eqnarray}
where $\varepsilon_{h}\,(h\leq g)$ denotes
the excursion (away from $0$)  with left-end point $h$ for the reflected process $\{\bar{X}(t)-X(t);t\geq0\}$, and $\zeta_{h}$ and $\bar{\varepsilon}_{h}$ denote its lifetime and excursion height, respectively.
Note that by \eqref{fluc.1},
\begin{eqnarray}\label{}
\hspace{-0.3cm}&&\hspace{-0.3cm}
\frac{\ell_{r}^{(q)}(x)}{\ell_{r}^{(q)}(a)}
\,n\left(\mathrm{e}^{-q \alpha_{a}^{+}(\varepsilon)}\mathrm{e}^{\lambda \left(a-\varepsilon\left(\alpha_{a}^{+}(\varepsilon)+r\right)\right)-
\psi\left(\lambda\right)r}\mathbf{1}_{\{\alpha_{a}^{+}(\varepsilon)< \zeta\}}\right)
\nonumber\\
\hspace{-0.3cm}&=&\hspace{-0.3cm}
\frac{\ell_{r}^{(q)}(x)\,\ell_{r}^{(q)\prime}(a)}{\left(\ell_{r}^{(q)}(a)\right)^{2}}
\left(\mathrm{e}^{\lambda a}-\left(\psi(\lambda)-q\right)
\left(\mathrm{e}^{\lambda a}\int_{0}^{a}W^{(q)}(z)\mathrm{e}^{-\lambda z}\mathrm{d}z+\int_{0}^{r}
\mathrm{e}^{-\psi(\lambda) s}\,\ell_{s}^{(q)}(a)\mathrm{d}s\right)\right)
\nonumber\\
\hspace{-0.3cm}&&\hspace{-0.3cm}
-\frac{\ell_{r}^{(q)}(x)}{\ell_{r}^{(q)}(a)}
\left(\lambda \mathrm{e}^{\lambda a}-\left(\psi(\lambda)-q\right)
\left(\lambda \mathrm{e}^{\lambda a}\int_{0}^{a}W^{(q)}(z)\mathrm{e}^{-\lambda z}\mathrm{d}z+W^{(q)}(a)+\int_{0}^{r}
\mathrm{e}^{-\psi(\lambda) s}\,\ell_{s}^{(q)\prime}(a)\mathrm{d}s\right)\right).
\nonumber
\end{eqnarray}
We thus obtain \eqref{n.2}.
\end{proof}

\vspace{0.3cm}
The following result gives an expression of the potential measure of the excursion process until time $\alpha_{a}^{+}$ under  the excursion measure.

\begin{prop}\label{lemma3}
For any $q,\lambda\in[0,\infty)$, $a,r\in(0,\infty)$ and any bounded differentiable function $f$, we have
\begin{eqnarray}\label{n.3}
\hspace{-0.3cm}&&\hspace{-0.3cm}
W^{(q)}(0)\,\mathrm{e}^{q r}
f(a)
+
\,n\left(\int_{0}^{\zeta}\mathrm{e}^{-q (t-r)}
f(a-\varepsilon\left(t\right))
\mathbf{1}_{\{\alpha_{a}^{+}(\varepsilon)> t-r\}}\mathrm{d}t\right)
\nonumber\\
\hspace{-0.3cm}&=&\hspace{-0.3cm}
\frac{\ell_{r}^{(q)\prime}(a)}{\ell_{r}^{(q)}(a)}
\left(\int_{0}^{r}\mathrm{e}^{q(r-s)}\mathbb{E}_{a}\left(f(X(s))\right)\mathrm{d}s
-\int_{0}^{a}W^{(q)}(a-z)\mathbb{E}_{z}\left(f(X(r))\right)\mathrm{d}z\right.
\nonumber\\
\hspace{-0.3cm}&&\hspace{-0.3cm}
\left.-\int_{0}^{r}\mathbb{E}_{}\left(f(X(r-s))\right)\ell_{s}^{(q)}(a)\mathrm{d}s\right)
\nonumber\\
\hspace{-0.3cm}&&\hspace{-0.3cm}
-
\int_{0}^{r}\mathrm{e}^{q(r-s)}\mathbb{E}_{}\left(f'(a+X(s))\right)
\mathrm{d}s
+\int_{0}^{a}W^{(q)\prime}(a-z)\mathbb{E}_{z}\left(f(X(r))\right)\mathrm{d}z
\nonumber\\
\hspace{-0.3cm}&&\hspace{-0.3cm}
+W^{(q)}(0+)\,\mathbb{E}_{a}\left(f(X(r))\right)
+\int_{0}^{r}\mathbb{E}_{}\left(f(X(r-s))\right)\ell_{s}^{(q)\prime}(a)\mathrm{d}s.
\end{eqnarray}
\end{prop}

\begin{proof}[Proof:]\,\,\,
Let $e_{q}$ be an exponentially distributed random variable with mean $1/q$ independent of $X$. For $q,\lambda\geq0$ and $r, b>0$ with $x\leq a$, we have
\begin{eqnarray}\label{h1h2.10}
\hspace{-0.3cm}&&\hspace{-0.3cm}
\int_{0}^{\infty}\mathrm{e}^{-q \left(t-r\right)}
\mathbb{E}_{x}\left(f(X(t));\,t<\kappa_{r}\wedge \tau_{a}^{+}\right)
\mathrm{d}t
\nonumber\\
\hspace{-0.3cm}&=&\hspace{-0.3cm}
\mathbb{E}_{x}\left(\int_{0}^{\kappa_{r}\wedge \tau_{a}^{+}}\mathrm{e}^{-q \left(t-r\right)}
f(X(t))\,
\mathrm{d}\left(\int_{0}^{t}\mathbf{1}_{\{X(s)=\bar{X}(s)\}}\mathrm{d}s\right)\right)
\nonumber\\
\hspace{-0.3cm}&&\hspace{-0.3cm}
+
\frac{1}{q}\mathbb{E}_{x}\left(\mathrm{e}^{q r}
f(X(e_{q}))\,\mathbf{1}_{\{e_{q}<\kappa_{r}\wedge \tau_{a}^{+}\}}\mathbf{1}_{\{X(e_{q})<\bar{X}(e_{q})\}}\right)
\nonumber\\
\hspace{-0.3cm}&:=&\hspace{-0.3cm}
h_{1}(x)+h_{2}(x).
\end{eqnarray}
Note that
$$\int_{0}^{t}\mathbf{1}_{\{X(s)=\bar{X}(s)\}}\mathrm{d}s=W^{(q)}(0) \,\bar{X}(t)$$
 and that  $X(t)=\bar{X}(t)$ implies $t=L^{-1}(L(t))$ a.s., the function $h_{1}(x)$ can be further expressed as follows.
\begin{eqnarray}\label{h1.10}
\hspace{-0.3cm}&&\hspace{-0.3cm}
W^{(q)}(0)\,\mathbb{E}_{x}\left(\int_{0}^{\infty}\mathrm{e}^{-q \left(L^{-1}(L(t))-r\right)}
f(x+L(t))
\mathbf{1}_{\{x+L(t)\leq a,\,L^{-1}(L(t))<\kappa_{r}\}}
\mathrm{d}L(t)\right)
\nonumber\\
\hspace{-0.3cm}&=&\hspace{-0.3cm}
W^{(q)}(0)\,\mathbb{E}_{x}\left(\int_{0}^{a-x}\mathrm{e}^{-q \left(L^{-1}(w)-r\right)}
f(x+w)
\mathbf{1}_{\{L^{-1}(w)<\kappa_{r}\}}
\mathrm{d}w\right)
\nonumber\\
\hspace{-0.3cm}&=&\hspace{-0.3cm}
W^{(q)}(0)\,\mathrm{e}^{q r}\int_{0}^{a-x}\mathbb{E}_{x}\left(
\mathrm{e}^{-q L^{-1}(w)}\mathbf{1}_{\{L^{-1}(w)<\kappa_{r}\}}\right)
f(x+w)
\mathrm{d}w
\nonumber\\
\hspace{-0.3cm}&=&\hspace{-0.3cm}
W^{(q)}(0)\,\mathrm{e}^{q r}\int_{x}^{a}
\frac{\ell^{(q)}_{r}(x)}{\ell_{r}^{(q)}(w)}\,
f(w)
\mathrm{d}w,
\end{eqnarray}
where \eqref{two.sid.ex.} is used in the final equality.

To further develop  $h_{2}(x)$, note  that
$$X(e_{q})<\bar{X}(e_{q})$$
if and only if there is an excursion  with left end point $g$  such that
$$e_{q}\in(g,g+\zeta_{g}).$$
Hence, by the compensation formula and the memoryless property of the exponential random variable, $h_{2}(x)$ can be rewritten as
\begin{eqnarray}\label{h2.10}
\hspace{-0.3cm}&&\hspace{-0.3cm}
\frac{1}{q}\mathbb{E}_{x}\left(\sum_{g}\mathrm{e}^{q r}
\prod\limits_{h<g}\mathbf{1}_{\{\alpha_{x+L(h)}^{+}(\varepsilon_{h})= \,\zeta_{h},\,x+L(g)\leq a\}}
\right.
\nonumber\\
\hspace{-0.3cm}&&\hspace{0.5cm}
\left.
\times f(x+L(g)-\varepsilon_{g}\left(e_{q}-g\right))
\mathbf{1}_{\{\alpha_{x+L(g)}^{+}(\varepsilon_{g})>e_{q}-g-r,\,0<e_{q}-g<\zeta_{g}\}}\right)
\nonumber\\
\hspace{-0.3cm}&=&\hspace{-0.3cm}
\frac{1}{q}\mathbb{E}_{x}\left(\sum_{g}\mathrm{e}^{-q (g-r)}
\prod\limits_{h<g}\mathbf{1}_{\{\alpha_{x+L(h)}^{+}(\varepsilon_{h})= \,\zeta_{h},\,x+L(g)\leq a\}}
\right.
\nonumber\\
\hspace{-0.3cm}&&\hspace{0.5cm}
\left.
\times f(x+L(g)-\varepsilon_{g}(e_{q}))
\mathbf{1}_{\{\alpha_{x+L(g)}^{+}(\varepsilon_{g})>e_{q}-r,\,e_{q}<\zeta_{g}\}}\right)
\nonumber\\
\hspace{-0.3cm}&=&\hspace{-0.3cm}
\frac{1}{q}\mathbb{E}_{x}\left(\int_{0}^{\infty}\mathrm{e}^{-q (w-r)}\prod\limits_{h<w}\mathbf{1}_{\{\alpha_{x+L(h)}^{+}(\varepsilon_{h})= \,\zeta_{h},\,x+L(w)\leq a\}}
\right.
\nonumber\\
\hspace{-0.3cm}&&\hspace{0.5cm}
\left.
\times \int_{\mathcal{E}}f(x+L(w)-\varepsilon\left(e_{q}\right))
\mathbf{1}_{\{\alpha_{x+L(w)}^{+}(\varepsilon)> e_{q}-r,\,e_{q}<\zeta\}}n(\,\mathrm{d}\varepsilon)\,\mathrm{d}L(w)\right)
\nonumber\\
\hspace{-0.3cm}&=&\hspace{-0.3cm}
\frac{1}{q}\mathbb{E}_{x}\left(\int_{x}^{a}\mathrm{e}^{-q (\tau_{w}^{+}-r)}\mathbf{1}_{\{\tau_{w}^{+}<\kappa_{r}\}}
\int_{\mathcal{E}}f(w-\varepsilon\left(e_{q}\right))
\mathbf{1}_{\{\alpha_{w}^{+}(\varepsilon)> e_{q}-r,\,e_{q}<\zeta\}}n(\,\mathrm{d}\varepsilon)\,\mathrm{d}w\right)
\nonumber\\
\hspace{-0.3cm}&=&\hspace{-0.3cm}
\int_{x}^{a}\mathbb{E}_{x}\left(\mathrm{e}^{-q \tau_{w}^{+}}\mathbf{1}_{\{\tau_{w}^{+}<\kappa_{r}\}}\right)
\,n\left(\int_{0}^{\zeta}\mathrm{e}^{-q (t-r)}
f(w-\varepsilon\left(t\right))
\mathbf{1}_{\{\alpha_{w}^{+}(\varepsilon)> t-r\}}\mathrm{d}t\right)\,\mathrm{d}w
\nonumber\\
\hspace{-0.3cm}&=&\hspace{-0.3cm}
\int_{x}^{a}
\frac{\ell^{(q)}_{r}(x)}{\ell_{r}^{(q)}(w)}
n\left(\int_{0}^{\zeta}\mathrm{e}^{-q (t-r)}
f(w-\varepsilon\left(t\right))
\mathbf{1}_{\{\alpha_{w}^{+}(\varepsilon)> t-r\}}\mathrm{d}t\right)\mathrm{d}w.
\end{eqnarray}

It follows from \eqref{h1h2.10}, \eqref{h1.10} and \eqref{h2.10} that
\begin{eqnarray}\label{21.10}
\hspace{-0.3cm}&&\hspace{-0.3cm}
\int_{0}^{\infty}\mathrm{e}^{-q \left(t-r\right)}
\mathbb{E}_{x}\left(f(X(t)); \,t<\kappa_{r}\wedge \tau_{a}^{+}\right)
\mathrm{d}t
\nonumber\\
\hspace{-0.3cm}&=&\hspace{-0.3cm}
W^{(q)}(0)\,\mathrm{e}^{q r}\int_{x}^{a}
\frac{\ell^{(q)}_{r}(x)}{\ell_{r}^{(q)}(w)}\,
f(w)
\mathrm{d}w
\nonumber\\
\hspace{-0.3cm}&&\hspace{-0.3cm}
+\int_{x}^{a}\frac{\ell^{(q)}_{r}(x)}{\ell_{r}^{(q)}(w)}
\,n\left(\int_{0}^{\zeta}\mathrm{e}^{-q (t-r)}
f(w-\varepsilon\left(t\right))
\mathbf{1}_{\{\alpha_{w}^{+}(\varepsilon)> t-r\}}\mathrm{d}t\right)\mathrm{d}w.
\end{eqnarray}
Meanwhile, by Theorem 4.4 of Loeffen et al. (2018) we know that
\begin{eqnarray}\label{22.10}
\hspace{-0.3cm}&&\hspace{-0.3cm}
\int_{0}^{\infty}\mathrm{e}^{-q \left(t-r\right)}
\mathbb{E}_{x}\left(f(X(t)); \,t<\kappa_{r}\wedge \tau_{a}^{+}\right)
\mathrm{d}t
\nonumber\\
\hspace{-0.3cm}&=&\hspace{-0.3cm}
\int_{0}^{r}\mathrm{e}^{q(r-s)}\mathbb{E}_{x}\left(f(X(s))\right)\mathrm{d}s
-\int_{0}^{x}W^{(q)}(x-z)\mathbb{E}_{z}\left(f(X(r))\right)\mathrm{d}z
\nonumber\\
\hspace{-0.3cm}&&\hspace{-0.3cm}
-\int_{0}^{r}\mathbb{E}_{}\left(f(X(r-s))\right)\ell_{s}^{(q)}(x)\mathrm{d}s
\nonumber\\
\hspace{-0.3cm}&&\hspace{-0.3cm}
-\frac{\ell_{r}^{(q)}(x)}{\ell_{r}^{(q)}(a)}
\left(\int_{0}^{r}\mathrm{e}^{q(r-s)}\mathbb{E}_{a}\left(f(X(s))\right)\mathrm{d}s
-\int_{0}^{a}W^{(q)}(a-z)\mathbb{E}_{z}\left(f(X(r))\right)\mathrm{d}z\right.
\nonumber\\
\hspace{-0.3cm}&&\hspace{-0.3cm}
\left.-\int_{0}^{r}\mathbb{E}_{}\left(f(X(r-s))\right)\ell_{s}^{(q)}(a)\mathrm{d}s\right).
\end{eqnarray}
Combining \eqref{21.10} and \eqref{22.10} we obtain (\ref{n.3}).
\end{proof}

\vspace{0.2cm}
\begin{rem}\label{rem0}
Replacing $f(x)$ with $f(x,a)$ in Proposition \ref{lemma3}, by similar arguments we have
\begin{eqnarray}\label{n.3''}
\hspace{-0.3cm}&&\hspace{-0.3cm}
W^{(q)}(0)\,\mathrm{e}^{q r}
f(a,a)
+
\,n\left(\int_{0}^{\zeta}\mathrm{e}^{-q (t-r)}
f(a-\varepsilon\left(t\right),a)
\mathbf{1}_{\{\alpha_{a}^{+}(\varepsilon)> t-r\}}\mathrm{d}t\right)
\nonumber\\
\hspace{-0.3cm}&=&\hspace{-0.3cm}
\frac{\ell_{r}^{(q)\prime}(a)}{\ell_{r}^{(q)}(a)}
\left(\int_{0}^{r}\mathrm{e}^{q(r-s)}\mathbb{E}_{a}\left(f(X(s),a)\right)\mathrm{d}s
-\int_{0}^{a}W^{(q)}(a-z)\mathbb{E}_{z}\left(f(X(r),a)\right)\mathrm{d}z\right.
\nonumber\\
\hspace{-0.3cm}&&\hspace{-0.3cm}
\left.-\int_{0}^{r}\mathbb{E}_{}\left(f(X(r-s),a)\right)\ell_{s}^{(q)}(a)\mathrm{d}s\right)
\nonumber\\
\hspace{-0.3cm}&&\hspace{-0.3cm}
-
\int_{0}^{r}\mathrm{e}^{q(r-s)}\mathbb{E}_{}\left(\frac{\partial}{\partial x}f(a+X(s),a)\right)
\mathrm{d}s
+\int_{0}^{a}W^{(q)\prime}(a-z)\mathbb{E}_{z}\left(f(X(r),a)\right)\mathrm{d}z
\nonumber\\
\hspace{-0.3cm}&&\hspace{-0.3cm}
\left.+W^{(q)}(0+)\,\mathbb{E}_{a}\left(f(X(r),a)\right)
+\int_{0}^{r}\mathbb{E}_{}\left(f(X(r-s),a)\right)\ell_{s}^{(q)\prime}(a)\mathrm{d}s\right)
.
\end{eqnarray}
\end{rem}

\vspace{0.2cm}
\begin{rem} \label{rem1}
Letting $f(x):=\mathrm{e}^{\lambda x-\psi\left(\lambda\right)r}$ in Proposition \ref{lemma3}, we have
\begin{eqnarray}\label{n.3'}
\hspace{-0.3cm}&&\hspace{-0.3cm}
n\left(\int_{0}^{\zeta}\mathrm{e}^{-q (t-r)}\mathrm{e}^{\lambda \left(a-\varepsilon\left(t\right)\right)-\psi\left(\lambda\right)r}\mathbf{1}_{\{\alpha_{a}^{+}(\varepsilon)> t-r\}}\,\mathrm{d}t\right)+
W^{(q)}(0)\,\mathrm{e}^{q r}
\mathrm{e}^{\lambda a-\psi\left(\lambda\right)r}
\nonumber\\
\hspace{-0.3cm}&=&\hspace{-0.3cm}
\frac{\ell_{r}^{(q)\prime}(a)}{\ell_{r}^{(q)}(a)}
\left(\frac{\mathrm{e}^{\lambda a}\left(1-\mathrm{e}^{-(\psi(\lambda)-q)r}\right)}{\psi(\lambda)-q}-
\mathrm{e}^{\lambda a}\int_{0}^{a}W^{(q)}(z)\mathrm{e}^{-\lambda z}\mathrm{d}z-\int_{0}^{r}
\mathrm{e}^{-\psi(\lambda) s}\,\ell_{s}^{(q)}(a)\mathrm{d}s\right)
\nonumber\\
\hspace{-0.3cm}&&\hspace{-0.3cm}
-
\frac{\lambda \mathrm{e}^{\lambda a}\left(1-\mathrm{e}^{-(\psi(\lambda)-q)r}\right)}{\psi(\lambda)-q}+
\lambda \mathrm{e}^{\lambda a}\int_{0}^{a}W^{(q)}(z)\mathrm{e}^{-\lambda z}\mathrm{d}z+W^{(q)}(a)+\int_{0}^{r}
\mathrm{e}^{-\psi(\lambda) s}\,\ell_{s}^{(q)\prime}(a)\mathrm{d}s
.\nonumber
\end{eqnarray}
\end{rem}

\vspace{0.3cm}

\section{Main results}
\setcounter{section}{4}
\setcounter{equation}{0}

In this section we present several results concerning the draw-down Parisian ruin.
The following first result solves a draw-down Parisian ruin based two-side exit problem. It generalizes Theorem 1 of Czarna and Renaud (2016).

\vspace{0.3cm}

\begin{thm}\label{thm00}
Given any $a$, let $\eta$ be another draw-down function such that $\eta(z)<\xi(z)< z$ for $z\leq a$.
For any $x\in (-\infty, a)$, we have
\begin{eqnarray}\label{two.sid.par.dra.dow.c}
\hspace{-0.3cm}&&\hspace{-0.3cm}
\mathbb{E}_{x}\left(\mathrm{e}^{-q \tau_{a}^{+}}\mathbf{1}_{\{\tau_{a}^{+}<\kappa_{r}^{\xi}\wedge \tau_{\eta}\}}\right)
\nonumber\\
\hspace{-0.3cm}&=&\hspace{-0.3cm}
\exp\left(-\int_{x}^{a}
\frac{W^{(q)\prime}(\overline{\xi}(w))\phi_{_{\Phi_{q}}}(\xi(w)-\eta(w),r)
+\chi^{(q)\prime}(\overline{\xi}(w),\xi(w)-\eta(w),r)}
{W^{(q)}(\overline{\xi}(w))\phi_{_{\Phi_{q}}}(\xi(w)-\eta(w),r)+\chi^{(q)}(\overline{\xi}(w),\xi(w)-\eta(w),r)}
\,\mathrm{d}w\right),\nonumber
\end{eqnarray}
where  the Laplace transforms of $\phi_{_{\Phi_{q}}}(y,r)$ and $\chi^{(q)}(x,y,r):=\mathrm{e}^{\Phi_{q}x}\chi_{_{\Phi_{q}}}(x,y,r)$ with respect to $r$, are given, respectively, by
\begin{eqnarray}
\label{Lap.4.1}
\hspace{-0.3cm}&&\hspace{-0.3cm}
\int_{0}^{\infty}\mathrm{e}^{-\theta r}\chi^{(q)}(x,y,r)\,\mathrm{d}r
\nonumber\\
\hspace{-0.3cm}&=&\hspace{-0.3cm}\frac{1}{\theta}\left(\frac{W_{y}^{(\theta+q,-\theta)}(x+y)}{W^{(\theta+q)}(y)}
-\frac{W^{(q)}(x)\mathrm{e}^{\Phi_{q}y}
\left(1+\theta\int_{0}^{y}\mathrm{e}^{-\Phi_{q}w}W^{(\theta+q)}(w)\mathrm{d}w\right)}{W^{(\theta+q)}(y)}\right)
,
\end{eqnarray}
and
\begin{eqnarray}
\label{Lap.4.2}
\int_{0}^{\infty}\mathrm{e}^{-\theta r}\phi_{_{\Phi_{q}}}(y,r)\,\mathrm{d}r
=
\frac{\mathrm{e}^{\Phi_{q}y}
\left(1+\theta\int_{0}^{y}\mathrm{e}^{-\Phi_{q}w}W^{(\theta+q)}(w)\mathrm{d}w\right)}
{\theta W^{(\theta+q)}(y)}
,
\end{eqnarray}
where $y\in(0,\infty)$ and  the derivative  of $\chi^{(q)}$ is taken on the first argument, and $\phi_{_{\Phi_{q}}}$ and $\chi_{_{\Phi_{q}}}$ play the roles of $\phi$ and $\chi$ for the process $(X,\mathbb{P}_{x}^{\Phi_{q}})$.
\end{thm}

\begin{proof}[Proof:]\,\,\,
By \eqref{n.pari.dra.dow.pro.c.0} and similar argument as in \eqref{pari.dra.dow.pro.c.0} we have
\begin{eqnarray}\label{pari.dra.dow.pro.}
\mathbb{P}_{x}\left(\tau_{a}^{+}<\kappa_{r}^{\xi}\wedge \tau_{\eta}\right)
\hspace{-0.3cm}&=&\hspace{-0.3cm}
\mathbb{E}_{x}\left(
\prod\limits_{t\leq a-x}\mathbf{1}_{\{\alpha_{\overline{\xi}(x+t)}^{+}(\varepsilon_{t})= \zeta(\varepsilon_{t}),\,\overline{\varepsilon}_{t}\leq \overline{\eta}(x+t)\}}\right)
\nonumber\\
\hspace{-0.3cm}&=&\hspace{-0.3cm}
\exp\left(-\int_{x}^{a}n\left(\alpha_{\overline{\xi}(w)}^{+}(\varepsilon)<\zeta \mbox{ or } \overline{\varepsilon}> \overline{\xi}(w)+\xi(w)-\eta(w)\right)\,\mathrm{d}w\right)
\nonumber\\
\hspace{-0.3cm}&=&\hspace{-0.3cm}
\exp\left(-\int_{x}^{a}
\frac{W^{\prime}(\overline{\xi}(w))\phi(\xi(w)-\eta(w),r)
+\chi^{\prime}(\overline{\xi}(w),\xi(w)-\eta(w),r)}
{W(\overline{\xi}(w))\phi(\xi(w)-\eta(w),r)+\chi(\overline{\xi}(w),\xi(w)-\eta(w),r)}
\,\mathrm{d}w\right),
\end{eqnarray}
where $\overline{\eta}(w):=w-\eta(w)$.
By \eqref{pari.dra.dow.pro.} together with a change of measure, one has
\begin{eqnarray}\label{renewal.iden.via.n}
\hspace{-0.3cm}&&\hspace{-0.3cm}
\mathbb{E}_{x}\left(\mathrm{e}^{-q \tau_{a}^{+}}\mathbf{1}_{\{\tau_{a}^{+}<\kappa_{r}^{\xi}\wedge \tau_{\eta}\}}\right)
\nonumber\\
\hspace{-0.3cm}&=&\hspace{-0.3cm}
\mathrm{e}^{-\Phi_{q}\left(a-x\right)}\mathbb{P}_{x}^{\Phi_{q}}\left(\tau_{a}^{+}<\kappa_{r}^{\xi}\wedge \tau_{\eta}\right)
\nonumber\\
\hspace{-0.3cm}&=&\hspace{-0.3cm}
\mathrm{e}^{-\Phi_{q}\left(a-x\right)}
\exp\left(-\int_{x}^{a}n_{_{\Phi_{q}}}\left(\alpha_{\overline{\xi}(w)}^{+}(\varepsilon)<\zeta \mbox{ or }
\overline{\varepsilon}> \overline{\eta}(w)\right)\,\mathrm{d}w\right)
\nonumber\\
\hspace{-0.3cm}&=&\hspace{-0.3cm}
\exp\left(-\Phi_{q}\left(a-x\right)-\int_{x}^{a}
\frac{W_{\Phi_{q}}^{\prime}(\overline{\xi}(w))\phi_{_{\Phi_{q}}}(\xi(w)-\eta(w),r)
+\chi_{_{\Phi_{q}}}^{\prime}(\overline{\xi}(w),\xi(w)-\eta(w),r)}
{W_{\Phi_{q}}(\overline{\xi}(w))\phi_{_{\Phi_{q}}}(\xi(w)-\eta(w),r)+\chi_{_{\Phi_{q}}}(\overline{\xi}(w),\xi(w)-\eta(w),r)}
\,\mathrm{d}w\right)
\nonumber\\
\hspace{-0.3cm}&=&\hspace{-0.3cm}
\exp\left(-\int_{x}^{a}
\frac{W^{(q)\prime}(\overline{\xi}(w))\phi_{_{\Phi_{q}}}(\xi(w)-\eta(w),r)
+\chi^{(q)\prime}(\overline{\xi}(w),\xi(w)-\eta(w),r)}
{W^{(q)}(\overline{\xi}(w))\phi_{_{\Phi_{q}}}(\xi(w)-\eta(w),r)+\chi^{(q)}(\overline{\xi}(w),\xi(w)-\eta(w),r)}
\,\mathrm{d}w\right),\nonumber
\end{eqnarray}
where
\[W_{\Phi_{q}}(x)=\mathrm{e}^{-\Phi_{q}x}W^{(q)}(x), \,\,\, \,\,\, \chi_{_{\Phi_{q}}}(x,y,r)=\mathrm{e}^{-\Phi_{q}x}\chi^{(q)}(x,y,r).\]
and $n_{_{\Phi_{q}}}$ represents the excursion measure under the new probability measure $\mathbb{P}_{x}^{\Phi_{q}}$.
\end{proof}

\vspace{0.3cm}
We next provide a version of Theorem \ref{thm00} for $\eta\equiv-\infty$.

 \begin{cor}\label{}
For $x\in (-\infty, a)$, we have
\begin{eqnarray}\label{newv.tw.si.}
\mbox{\,\,\,\,\,}\mathbb{E}_{x}\left(\mathrm{e}^{-q \tau_{a}^{+}}\mathbf{1}_{\{\tau_{a}^{+}<\kappa_{r}^{\xi}\}}\right)
=\exp\left(-\int_{x}^{a} \frac{\ell^{(q)\prime}_{r}(\overline{\xi}(w))}{\ell_{r}^{(q)}(\overline{\xi}(w))}\,\mathrm{d}w\right).
\end{eqnarray}
 \end{cor}

\begin{proof}
By definition we have
$$W^{(\theta+q,-\theta)}_{y}(x+y)=W^{(\theta+q)}(x+y)-\theta\int_{0}^{x}W^{(q)}(w)W^{(\theta+q)}(x+y-w)\,\mathrm{d}w,$$
which together with \eqref{t.s.limit} yields
\begin{eqnarray}
\label{rem4.1.lap.0}
\lim\limits_{y\uparrow\infty}\frac{W^{(\theta+q,-\theta)}_{y}(x+y)}{\theta W^{(\theta+q)}(y)}
\hspace{-0.3cm}&=&\hspace{-0.3cm}
\mathrm{e}^{\Phi_{\theta+q}x}
\left(\frac{1}{\theta}-\int_{0}^{x}W^{(q)}(w)\,
\mathrm{e}^{-\Phi_{\theta+q}w}\,\mathrm{d}w\right)
\nonumber\\
\hspace{-0.3cm}&=&\hspace{-0.3cm}
\mathrm{e}^{\Phi_{\theta+q}x}
\left(\frac{1}{\theta}-\int_{0}^{\infty}W^{(q)}(w)\,
\mathrm{e}^{-\Phi_{\theta+q}w}\,\mathrm{d}w
+\int_{x}^{\infty}W^{(q)}(w)\,
\mathrm{e}^{-\Phi_{\theta+q}w}\,\mathrm{d}w\right)
\nonumber\\
\hspace{-0.3cm}&=&\hspace{-0.3cm}
\int_{0}^{\infty}W^{(q)}(x+w)\,\mathrm{e}^{-\Phi_{\theta+q} w}\,\mathrm{d}w
,
\end{eqnarray}
which coincides with the Laplace transform (in $r$) of $\,\mathrm{e}^{-qr}\ell_{r}^{(q)}(x)$ as follows.
\begin{eqnarray}\label{la.id.2}
\int_{0}^{\infty}\mathrm{e}^{-\theta r}\mathrm{e}^{-qr}\ell_{r}^{(q)}(x)\,\mathrm{d}r
\hspace{-0.3cm}&=&\hspace{-0.3cm}
\int_{0}^{\infty}\mathrm{e}^{-(\theta+q)r}\int_{0}^{\infty}W^{(q)}(x+z)
\frac{z}{r}\mathbb{P}\left(X(r)\in\mathrm{d}z\right)\,\mathrm{d}r
\nonumber\\
\hspace{-0.3cm}&=&\hspace{-0.3cm}
\int_{0}^{\infty}\mathrm{e}^{-(\theta+q)r}\int_{0}^{\infty}W^{(q)}(x+z)
\mathbb{P}(\tau_{z}^{+}\in\mathrm{d}r)\mathrm{d}z
\nonumber\\
\hspace{-0.3cm}&=&\hspace{-0.3cm}
\int_{0}^{\infty}W^{(q)}(x+z)
\mathbb{E}(\mathrm{e}^{-(\theta+q) \tau_{z}^{+}})\mathrm{d}z
\nonumber\\
\hspace{-0.3cm}&=&\hspace{-0.3cm}
\int_{0}^{\infty}W^{(q)}(x+w)\,\mathrm{e}^{-\Phi_{\theta+q} w}\,\mathrm{d}w
,
\end{eqnarray}
where we have used the Kendall's identity \[\frac{z}{r}\mathbb{P}\left(X(r)\in\mathrm{d}z\right)\,\mathrm{d}r=\mathbb{P}(\tau_{z}^{+}\in\mathrm{d}r)\mathrm{d}z,\,z,r\geq0.\]
From \eqref{Lap.4.1} and \eqref{Lap.4.2} one knows that
\begin{eqnarray}
\label{rem4.1.lap.1}
\int_{0}^{\infty}\mathrm{e}^{-\theta r}\left(W^{(q)}(x)\phi_{_{\Phi_{q}}}(y,r)+\chi^{(q)}(x,y,r)\right)\mathrm{d}r
=\frac{W_{y}^{(\theta+q,-\theta)}(x+y)}{\theta W^{(\theta+q)}(y)}
.
\end{eqnarray}
Combining \eqref{rem4.1.lap.0}, \eqref{la.id.2} and \eqref{rem4.1.lap.1}, 
 one can conclude
$$\lim\limits_{y\uparrow\infty}\left(W^{(q)}(x)\phi_{_{\Phi_{q}}}(y,r)+\chi^{(q)}(x,y,r)\right)=\mathrm{e}^{-qr}\ell_{r}^{(q)}(x)
.$$
By the same arguments, we have
$$\lim\limits_{y\uparrow\infty}\left(W^{(q)\prime}(x)\phi_{_{\Phi_{q}}}(y,r)+\chi^{(q)\prime}(x,y,r)\right)
=\mathrm{e}^{-qr}\ell_{r}^{(q)\prime}(x)
.$$
Hence, we have
\begin{eqnarray}
\hspace{-0.3cm}&&\hspace{-0.3cm}
\lim\limits_{c\uparrow\infty}\frac{W^{(q)\prime}(\overline{\xi}(w))\phi_{_{\Phi_{q}}}(c+\xi(w),r)
+\chi^{(q)\prime}(\overline{\xi}(w),c+\xi(w),r)}
{W^{(q)}(\overline{\xi}(w))\phi_{_{\Phi_{q}}}(c+\xi(w),r)+\chi^{(q)}(\overline{\xi}(w),c+\xi(w),r)}
=\frac{\ell_{r}^{(q)\prime}(\overline{\xi}(w))}{\ell_{r}^{(q)}(\overline{\xi}(w))}.
\nonumber
\end{eqnarray}
It then follows easily from Theorem \ref{thm00} that
\begin{eqnarray}
\mbox{\,\,\,\,\,}\mathbb{E}_{x}\left(\mathrm{e}^{-q \tau_{a}^{+}}\mathbf{1}_{\{\tau_{a}^{+}<\kappa_{r}^{\xi}\}}\right)
=\lim\limits_{c\uparrow\infty}\mathbb{E}_{x}\left(\mathrm{e}^{-q \tau_{a}^{+}}\mathbf{1}_{\{\tau_{a}^{+}<\kappa_{r}^{\xi}\wedge \tau_{-c}^{-}\}}\right)
=\exp\left(-\int_{x}^{a} \frac{\ell^{(q)\prime}_{r}(\overline{\xi}(w))}{\ell_{r}^{(q)}(\overline{\xi}(w))}\,\mathrm{d}w\right),
\nonumber
\end{eqnarray}
which is the desired result.
\end{proof}

\vspace{0.3cm}
\begin{rem}
Letting $\xi\equiv0$ in \eqref{newv.tw.si.}, one recovers (18) of Czarna and Palmowski (2014)
\begin{eqnarray*}\label{}
\mathbb{E}_{x}\left(\mathrm{e}^{-q \tau_{a}^{+}}\mathbf{1}_{\{\tau_{a}^{+}<\kappa_{r}\}}\right)
\hspace{-0.3cm}&=&\hspace{-0.3cm}
\exp\left(-\int_{x}^{a} \frac{\ell^{(q)\prime}_{r}(w)}{\ell_{r}^{(q)}(w)}\,\mathrm{d}w\right)
=\frac{\ell^{(q)}_{r}(x)}{\ell_{r}^{(q)}(a)}.\nonumber
\end{eqnarray*}

Letting $\xi(x)=kx-d$ with $k\in(-\infty,1)$ and $d\in (0,\infty)$ in \eqref{newv.tw.si.}, we have
\begin{eqnarray}
\label{}
\mathbb{E}_{x}\left(\mathrm{e}^{-q\tau_{a}^{+}}
\mathbf{1}_{\{\tau_{a}^{+}<\kappa_{r}^{\xi}\}}\right)
\hspace{-0.3cm}&=&\hspace{-0.3cm}
\left(\frac{\ell_{r}^{(q)}((1-k)x+d)}{ \ell_{r}^{(q)}((1-k)a+d)}\right)^{\frac{1}{1-k}}.\nonumber
\end{eqnarray}

By  \eqref{newv.tw.si.}, one can also obtain the draw-down Parisian ruin probability
\begin{eqnarray}\label{par.dra.dow.pro.}
\mathbb{P}_{x}\left(\kappa_{r}^{\xi}<\infty\right)
=1-\exp\left(-\int_{x}^{\infty} \frac{\ell_{r}^{\prime}(\overline{\xi}(z))}{\ell_{r}(\overline{\xi}(z))}\,\mathrm{d}z\right).
\nonumber
\end{eqnarray}
\end{rem}

\vspace{0.3cm}
The following result presents the joint Laplace transform involving the draw-down Parisian ruin time, the position of $X$ at the draw-down Parisian ruin time and its running supremum until the draw-down Parisian ruin time. It generalizes Theorem 3.1 in Loeffen et al. (2018).

\begin{thm}\label{thm2}
For any  $q,\lambda\in[0,\infty)$, $x\in(-\infty,\infty)$, $a \geq x$ and any bounded measurable function $\varphi:\,(-\infty,\infty)\mapsto (-\infty,\infty)$, we have
\begin{eqnarray}\label{joint.Lap.}
\hspace{-0.3cm}&&\hspace{-0.3cm}
\mathbb{E}_{x}\left(\mathrm{e}^{-q \left(\kappa_{r}^{\xi}-r\right)}\mathrm{e}^{\lambda X(\kappa_{r}^{\xi})-\psi\left(\lambda\right)r}\varphi(\bar{X}(\kappa_{r}^{\xi}))\mathbf{1}_{\{\kappa_{r}^{\xi}<\tau_{a}^{+}\}}\right)
\nonumber\\
\hspace{-0.3cm}&=&\hspace{-0.3cm}
\int_{x}^{a}
\mathrm{e}^{\lambda \xi(w)}\varphi(w)\exp\left(-\int_{x}^{w} \frac{\ell^{(q)\prime}_{r}(\overline{\xi}(z))}{\ell_{r}^{(q)}(\overline{\xi}(z))}\,\mathrm{d}z\right)
\left[
\frac{\ell_{r}^{(q)\prime}(\overline{\xi}(w))}{\ell_{r}^{(q)}(\overline{\xi}(w))}
\left(\mathrm{e}^{\lambda \overline{\xi}(w)}-\left(\psi(\lambda)-q\right)\right.\right.
\nonumber\\
\hspace{-0.3cm}&&\hspace{-0.3cm}
\times
\left.\left.\left(\mathrm{e}^{\lambda \overline{\xi}(w)}\int_{0}^{\overline{\xi}(w)}W^{(q)}(z)\mathrm{e}^{-\lambda z}\mathrm{d}z+\int_{0}^{r}
\mathrm{e}^{-\psi(\lambda) s}\,\ell_{s}^{(q)}(\overline{\xi}(w))\mathrm{d}s\right)\right)\right.
\nonumber\\
\hspace{-0.3cm}&&\hspace{-0.3cm}
\left.-
\lambda \mathrm{e}^{\lambda \overline{\xi}(w)}+\left(\psi(\lambda)-q\right)
\left(\lambda \mathrm{e}^{\lambda \overline{\xi}(w)}\int_{0}^{\overline{\xi}(w)}W^{(q)}(z)\mathrm{e}^{-\lambda z}\mathrm{d}z
\right.\right.
\nonumber\\
\hspace{-0.3cm}&&\hspace{-0.3cm}
\left.\left.
+W^{(q)}(\overline{\xi}(w))+\int_{0}^{r}
\mathrm{e}^{-\psi(\lambda) s}\,\ell_{s}^{(q)\prime}(\overline{\xi}(w))\mathrm{d}s\right)\right]\,\mathrm{d}w.
\end{eqnarray}
\end{thm}

\begin{proof}[Proof:]\,\,\,
By \eqref{newv.tw.si.} and  the compensation formula, we have
\begin{eqnarray}\label{}
\hspace{-0.3cm}&&\hspace{-0.3cm}
\mathbb{E}_{x}\left(\mathrm{e}^{-q \left(\kappa_{r}^{\xi}-r\right)}\mathrm{e}^{\lambda X(\kappa_{r}^{\xi})-\psi\left(\lambda\right)r}\varphi(\bar{X}(\kappa_{r}^{\xi}))\mathbf{1}_{\{\kappa_{r}^{\xi}<\tau_{a}^{+}\}}\right)
\nonumber\\
\hspace{-0.3cm}&=&\hspace{-0.3cm}
\mathbb{E}_{x}\left(\sum_{g}\mathrm{e}^{-q g}\varphi(x+L(g))\prod\limits_{h<g}\mathbf{1}_{\{\alpha_{\overline{\xi}(x+L(h))}^{+}(\varepsilon_{h})= \,\zeta_{h},\,x+L(g)\leq a\}}
\mathrm{e}^{-q \alpha_{\overline{\xi}(x+L(g))}^{+}(\varepsilon_{g})}\right.
\nonumber\\
\hspace{-0.3cm}&&\hspace{0.5cm}
\left.
\times\mathrm{e}^{\lambda \left(x+L(g)-\varepsilon_{g}\left(\alpha_{\overline{\xi}(x+L(g))}^{+}(\varepsilon_{g})+r\right)\right)
-\psi\left(\lambda\right)r}\mathbf{1}_{\{\alpha_{\overline{\xi}(x+L(g))}^{+}(\varepsilon_{g})< \zeta_{g}\}}\right)
\nonumber\\
\hspace{-0.3cm}&=&\hspace{-0.3cm}
\mathbb{E}_{x}\left(\int_{0}^{\infty}\mathrm{e}^{-q w}\varphi(x+L(w))\prod\limits_{h<w}\mathbf{1}_{\{\alpha_{\overline{\xi}(x+L(h))}^{+}(\varepsilon_{h})= \,\zeta_{h},\,x+L(w)\leq a\}}
\int_{\mathcal{E}}\mathrm{e}^{-q \alpha_{\overline{\xi}(x+L(w))}^{+}(\varepsilon)}\right.
\nonumber\\
\hspace{-0.3cm}&&\hspace{0.5cm}
\left.
\times\mathrm{e}^{\lambda \left(x+L(w)-\varepsilon\left(\alpha_{\overline{\xi}(x+L(w))}^{+}(\varepsilon)+r\right)\right)
-\psi\left(\lambda\right)r}\mathbf{1}_{\{\alpha_{\overline{\xi}(x+L(w))}^{+}(\varepsilon)< \zeta\}}\,n(\,\mathrm{d}\varepsilon)\,\mathrm{d}L(w)\right)
\nonumber\\
\hspace{-0.3cm}&=&\hspace{-0.3cm}
\mathbb{E}_{x}\left(\int_{0}^{a-x}\mathrm{e}^{-q L^{-1}(w-)}\varphi(x+w)
\prod\limits_{h<L^{-1}(w-)}\mathbf{1}_{\{\alpha_{\overline{\xi}(x+L(h))}^{+}(\varepsilon_{h})= \,\zeta_{h}\}}
\right.
\nonumber\\
\hspace{-0.3cm}&&\hspace{0.5cm}
\left.
\times\int_{\mathcal{E}}\mathrm{e}^{-q \alpha_{\overline{\xi}(x+w)}^{+}(\varepsilon)}\mathrm{e}^{\lambda \left(x+w-\varepsilon\left(\alpha_{\overline{\xi}(x+w)}^{+}(\varepsilon)+r\right)\right)
-\psi\left(\lambda\right)r}\mathbf{1}_{\{\alpha_{\overline{\xi}(x+w)}^{+}(\varepsilon)< \zeta\}}\,n(\,\mathrm{d}\varepsilon)\,\mathrm{d}w\right)
\nonumber\\
\hspace{-0.3cm}&=&\hspace{-0.3cm}
\int_{x}^{a}\varphi(w)\mathbb{E}_{x}\left(\mathrm{e}^{-q \tau_{w}^{+}}\mathbf{1}_{\{\tau_{w}^{+}<\kappa_{r}^{\xi}\}}\right)
\,n\left(\mathrm{e}^{-q \alpha_{\overline{\xi}(w)}^{+}(\varepsilon)}\mathrm{e}^{\lambda \left(w-\varepsilon\left(\alpha_{\overline{\xi}(w)}^{+}(\varepsilon)+r\right)\right)
-\psi\left(\lambda\right)r}\mathbf{1}_{\{\alpha_{\overline{\xi}(w)}^{+}(\varepsilon)< \zeta\}}\right)\,\mathrm{d}w
\nonumber\\
\hspace{-0.3cm}&=&\hspace{-0.3cm}
\int_{x}^{a}
\mathrm{e}^{\lambda \xi(w)}\varphi(w)\exp\left(-\int_{x}^{w} \frac{\ell^{(q)\prime}_{r}(\overline{\xi}(z))}{\ell_{r}^{(q)}(\overline{\xi}(z))}\,\mathrm{d}z\right)
\nonumber\\
\hspace{-0.3cm}&&\hspace{0.5cm}
\times n\left(\mathrm{e}^{-q \alpha_{\overline{\xi}(w)}^{+}(\varepsilon)}\mathrm{e}^{\lambda \left(\overline{\xi}(w)-\varepsilon\left(\alpha_{\overline{\xi}(w)}^{+}(\varepsilon)+r\right)\right)
-\psi\left(\lambda\right)r}\mathbf{1}_{\{\alpha_{\overline{\xi}(w)}^{+}(\varepsilon)< \zeta\}}\right)\,\mathrm{d}w,\nonumber
\end{eqnarray}
which together with \eqref{n.2} yields \eqref{joint.Lap.}.
\end{proof}

\vspace{0.3cm}
The following result gives
the potential measure of $X$ involving the draw-down Parisian ruin time.
It generalizes Theorem 4.4 in Loeffen et al. (2018).

\begin{thm}\label{thm3}
For any $q,\lambda\geq0$, $r>0$, $a\geq x$ and any bounded bivariate function $f(x,y)$ that is differentiable with respect to $x$, we have
\begin{eqnarray}\label{pot.mea..1'}
\hspace{-0.3cm}&&\hspace{-0.3cm}
\int_{0}^{\infty}\mathrm{e}^{-q \left(t-r\right)}
\mathbb{E}_{x}\left(f(X(t),\bar{X}(t)); \,t<\kappa_{r}^{\xi}\wedge \tau_{a}^{+}\right)
\mathrm{d}t
\nonumber\\
\hspace{-0.3cm}&=&\hspace{-0.3cm}
\int_{x}^{a}
\exp\left(-\int_{x}^{w} \frac{\ell^{(q)\prime}_{r}(\overline{\xi}(z))}{\ell_{r}^{(q)}(\overline{\xi}(z))}\,\mathrm{d}z\right)
\left[\frac{\ell_{r}^{(q)\prime}(\overline{\xi}(w))}{\ell_{r}^{(q)}(\overline{\xi}(w))}
\left(\int_{0}^{r}\mathrm{e}^{q(r-s)}\mathbb{E}_{}\left(f(w+X(s),w)\right)\mathrm{d}s
\right.\right.
\nonumber\\
\hspace{-0.3cm}&&\hspace{-0.3cm}
\left.-\int_{0}^{\overline{\xi}(w)}W^{(q)}(\overline{\xi}(w)-z)
\mathbb{E}_{}\left(f(z+\xi(w)+X(r),w)\right)\mathrm{d}z
\right.
\nonumber\\
\hspace{-0.3cm}&&\hspace{-0.3cm}
\left.-\int_{0}^{r}\mathbb{E}_{}\left(f(\xi(w)+X(r-s),w)\right)\ell_{s}^{(q)}(\overline{\xi}(w))\mathrm{d}s\right)
\nonumber\\
\hspace{-0.3cm}&&\hspace{-0.3cm}
-
\int_{0}^{r}\mathrm{e}^{q(r-s)}\mathbb{E}_{}\left(\frac{\partial}{\partial x}f(w+X(s),w)\right)
\mathrm{d}s
+\int_{0}^{\overline{\xi}(w)}W^{(q)\prime}(\overline{\xi}(w)-z)\mathbb{E}_{}\left(f(z+\xi(w)+X(r),w)\right)\mathrm{d}z
\nonumber\\
\hspace{-0.3cm}&&\hspace{-0.3cm}
\left.\left.+W^{(q)}(0+)\,\mathbb{E}_{}\left(f(w+X(r),w)\right)
+\int_{0}^{r}\mathbb{E}_{}\left(f(\xi(w)+X(r-s),w)\right)\ell_{s}^{(q)\prime}(\overline{\xi}(w))\mathrm{d}s\right)\right]
\,\mathrm{d}w.\nonumber
\end{eqnarray}
\end{thm}

\begin{proof}[Proof:]\,\,\,
For $q,\lambda\geq0$ and $r>0$ with $ a\geq x$, we have
\begin{eqnarray}\label{h1h2.1}
\hspace{-0.3cm}&&\hspace{-0.3cm}
\int_{0}^{\infty}\mathrm{e}^{-q \left(t-r\right)}
\mathbb{E}_{x}\left(f(X(t),\bar{X}(t));\,t<\kappa_{r}^{\xi}\wedge \tau_{a}^{+}\right)
\mathrm{d}t
\nonumber\\
\hspace{-0.3cm}&=&\hspace{-0.3cm}
\mathbb{E}_{x}\left(\int_{0}^{\kappa_{r}^{\xi}\wedge \tau_{a}^{+}}\mathrm{e}^{-q \left(t-r\right)}
f(X(t),\bar{X}(t))\,
\mathrm{d}\left(\int_{0}^{t}\mathbf{1}_{\{X(s)=\bar{X}(s)\}}\mathrm{d}s\right)\right)
\nonumber\\
\hspace{-0.3cm}&&\hspace{-0.3cm}
+
\frac{1}{q}\,\mathbb{E}_{x}\left(\mathrm{e}^{qr}f(X(e_{q}),\bar{X}(e_{q}))\,\mathbf{1}_{\{X(e_{q})<\bar{X}(e_{q}), \,e_{q}<\kappa_{r}^{\xi}\wedge \tau_{a}^{+}\}}
\right)
\nonumber\\
\hspace{-0.3cm}&:=&\hspace{-0.3cm}
I_{1}(x)+I_{2}(x).
\end{eqnarray}
Note that $\int_{0}^{t}\mathbf{1}_{\{X(s)=\bar{X}(s)\}}\mathrm{d}s=W^{(q)}(0) \,\bar{X}(t)$ and that  $X(t)=\bar{X}(t)$ implies $t=L^{-1}(L(t))$ a.s. The function $I_{1}(x)$ can be rewritten as
\begin{eqnarray}\label{I1}
\hspace{-0.3cm}&&\hspace{-0.3cm}
W^{(q)}(0)\,\mathbb{E}_{x}\left(\int_{0}^{\infty}\mathrm{e}^{-q \left(L^{-1}(L(t))-r\right)}
f(x+L(t),x+L(t))
\mathbf{1}_{\{x+L(t)\leq a,\,L^{-1}(L(t))<\kappa_{r}^{\xi}\}}
\mathrm{d}L(t)\right)
\nonumber\\
\hspace{-0.3cm}&=&\hspace{-0.3cm}
W^{(q)}(0)\,\mathbb{E}_{x}\left(\int_{0}^{a-x}\mathrm{e}^{-q \left(L^{-1}(w)-r\right)}
f(x+w,x+w)
\mathbf{1}_{\{L^{-1}(w)<\kappa_{r}^{\xi}\}}
\mathrm{d}w\right)
\nonumber\\
\hspace{-0.3cm}&=&\hspace{-0.3cm}
W^{(q)}(0)\,\mathrm{e}^{q r}\int_{0}^{a-x}\mathbb{E}_{x}\left(
\mathrm{e}^{-q L^{-1}(w)}\mathbf{1}_{\{L^{-1}(w)<\kappa_{r}^{\xi}\}}\right)
f(x+w,x+w)
\mathrm{d}w
\nonumber\\
\hspace{-0.3cm}&=&\hspace{-0.3cm}
W^{(q)}(0)\,\mathrm{e}^{q r}\int_{x}^{a}
\exp\left(-\int_{x}^{w} \frac{\ell^{(q)\prime}_{r}(\overline{\xi}(z))}{\ell_{r}^{(q)}(\overline{\xi}(z))}\,\mathrm{d}z\right)\,
f(\overline{\xi}(w)+\xi(w),\overline{\xi}(w)+\xi(w))
\mathrm{d}w,
\end{eqnarray}
where \eqref{two.sid.ex.} is used in the last equality.
Using the compensation formula, the function $I_{2}(x)$ can be rewritten as
\begin{eqnarray}\label{I2}
\hspace{-0.3cm}&&\hspace{-0.3cm}
\frac{1}{q}\mathbb{E}_{x}\left(\sum_{g}\mathrm{e}^{q r}
\prod\limits_{h<g}\mathbf{1}_{\{\alpha_{\overline{\xi}(x+L(h))}^{+}(\varepsilon_{h})= \,\zeta_{h},\,x+L(g)\leq a\}}
\right.
\nonumber\\
\hspace{-0.3cm}&&\hspace{0.5cm}
\left.
\times f(x+L(g)-\varepsilon_{g}\left(e_{q}-g\right),x+L(g))
\mathbf{1}_{\{\alpha_{\overline{\xi}(x+L(g))}^{+}(\varepsilon_{g})>e_{q}-g-r,\,0<e_{q}-g<\zeta_{g}\}}\right)
\nonumber\\
\hspace{-0.3cm}&=&\hspace{-0.3cm}
\frac{1}{q}\mathbb{E}_{x}\left(\sum_{g}\mathrm{e}^{-q (g-r)}
\prod\limits_{h<g}\mathbf{1}_{\{\alpha_{\overline{\xi}(x+L(h))}^{+}(\varepsilon_{h})= \,\zeta_{h},\,x+L(g)\leq a\}}
\right.
\nonumber\\
\hspace{-0.3cm}&&\hspace{0.5cm}
\left.
\times f(x+L(g)-\varepsilon_{g}(e_{q}),x+L(g))
\mathbf{1}_{\{\alpha_{\overline{\xi}(x+L(g))}^{+}(\varepsilon_{g})>e_{q}-r,\,e_{q}<\zeta_{g}\}}\right)
\nonumber\\
\hspace{-0.3cm}&=&\hspace{-0.3cm}
\frac{1}{q}\mathbb{E}_{x}\left(\int_{0}^{\infty}\mathrm{e}^{-q (w-r)}\prod\limits_{h<w}\mathbf{1}_{\{\alpha_{\overline{\xi}(x+L(h))}^{+}(\varepsilon_{h})= \,\zeta_{h},\,x+L(w)\leq a\}}
\right.
\nonumber\\
\hspace{-0.3cm}&&\hspace{0.5cm}
\left.
\times \int_{\mathcal{E}}f(x+L(w)-\varepsilon\left(e_{q}\right),x+L(w))
\mathbf{1}_{\{\alpha_{\overline{\xi}(x+L(w))}^{+}(\varepsilon)> e_{q}-r,\,e_{q}<\zeta\}}n(\,\mathrm{d}\varepsilon)\,\mathrm{d}L(w)\right)
\nonumber\\
\hspace{-0.3cm}&=&\hspace{-0.3cm}
\frac{1}{q}\mathbb{E}_{x}\left(\int_{x}^{a}\mathrm{e}^{-q (\tau_{w}^{+}-r)}\mathbf{1}_{\{\tau_{w}^{+}<\kappa_{r}^{\xi}\}}
\int_{\mathcal{E}}f(w-\varepsilon\left(e_{q}\right),w)
\mathbf{1}_{\{\alpha_{\overline{\xi}(w)}^{+}(\varepsilon)> e_{q}-r,\,e_{q}<\zeta\}}n(\,\mathrm{d}\varepsilon)\,\mathrm{d}w\right)
\nonumber\\
\hspace{-0.3cm}&=&\hspace{-0.3cm}
\int_{x}^{a}
\exp\left(-\int_{x}^{w} \frac{\ell^{(q)\prime}_{r}(\overline{\xi}(z))}{\ell_{r}^{(q)}(\overline{\xi}(z))}\,\mathrm{d}z\right)
\nonumber\\
\hspace{-0.3cm}&&\hspace{0.5cm}
\times n\left(\int_{0}^{\zeta}\mathrm{e}^{-q (t-r)}
f(\overline{\xi}(w)-\varepsilon\left(t\right)+\xi(w),\overline{\xi}(w)+\xi(w))
\mathbf{1}_{\{\alpha_{\overline{\xi}(w)}^{+}(\varepsilon)> t-r\}}\mathrm{d}t\right)\mathrm{d}w.
\end{eqnarray}
Combining \eqref{h1h2.1}, \eqref{I1}, \eqref{I2} and \eqref{n.3''} leads to the desired result.
\end{proof}

\vspace{0.3cm}
We have the following version of Theorem \ref{thm3} when $f$ is independent of $y$.

\begin{cor}\label{cor3.1}
For any $q,\lambda\geq0$, $r>0$, $a\geq x$ and bounded differentiable function $f$, we have
\begin{eqnarray}\label{pot.mea..1}
\hspace{-0.3cm}&&\hspace{-0.3cm}
\int_{0}^{\infty}\mathrm{e}^{-q \left(t-r\right)}
\mathbb{E}_{x}\left(f(X(t)); \,t<\kappa_{r}^{\xi}\wedge \tau_{a}^{+}\right)
\mathrm{d}t
\nonumber\\
\hspace{-0.3cm}&=&\hspace{-0.3cm}
\int_{x}^{a}
\exp\left(-\int_{x}^{w} \frac{\ell^{(q)\prime}_{r}(\overline{\xi}(z))}{\ell_{r}^{(q)}(\overline{\xi}(z))}\,\mathrm{d}z\right)
\left[\frac{\ell_{r}^{(q)\prime}(\overline{\xi}(w))}{\ell_{r}^{(q)}(\overline{\xi}(w))}
\left(\int_{0}^{r}\mathrm{e}^{q(r-s)}\mathbb{E}_{}\left(f(w+X(s))\right)\mathrm{d}s
\right.\right.
\nonumber\\
\hspace{-0.3cm}&&\hspace{-0.3cm}
\left.-\int_{0}^{\overline{\xi}(w)}W^{(q)}(\overline{\xi}(w)-z)\mathbb{E}_{}\left(f(z+\xi(w)+X(r))\right)\mathrm{d}z
\right.
\nonumber\\
\hspace{-0.3cm}&&\hspace{-0.3cm}
\left.-\int_{0}^{r}\mathbb{E}_{}\left(f(\xi(w)+X(r-s))\right)\ell_{s}^{(q)}(\overline{\xi}(w))\mathrm{d}s\right)
\nonumber\\
\hspace{-0.3cm}&&\hspace{-0.3cm}
-
\int_{0}^{r}\mathrm{e}^{q(r-s)}\mathbb{E}_{}\left(f'(w+X(s))\right)
\mathrm{d}s
+\int_{0}^{\overline{\xi}(w)}W^{(q)\prime}(\overline{\xi}(w)-z)\mathbb{E}_{}\left(f(z+\xi(w)+X(r))\right)\mathrm{d}z
\nonumber\\
\hspace{-0.3cm}&&\hspace{-0.3cm}
\left.\left.+W^{(q)}(0+)\,\mathbb{E}_{}\left(f(w+X(r))\right)
+\int_{0}^{r}\mathbb{E}_{}\left(f(\xi(w)+X(r-s))\right)\ell_{s}^{(q)\prime}(\overline{\xi}(w))\mathrm{d}s\right)\right]
\,\mathrm{d}w.\nonumber
\end{eqnarray}
\end{cor}

\vspace{0.3cm}
The following result gives the Laplace transform of the potential measure of $X$ killed upon up-crossing $a\,(\geq x)$ or draw-down Parisian ruin.

\begin{cor}\label{cor3.2}
For any $q,\lambda\geq0$, $r>0$ and $a \geq x$, we have
\begin{eqnarray}\label{Lap.pot.mea.}
\hspace{-0.3cm}&&\hspace{-0.3cm}
\mathbb{E}_{x}\left(\int_{0}^{\kappa_{r}^{\xi}\wedge \tau_{a}^{+}}\mathrm{e}^{-q \left(t-r\right)}\mathrm{e}^{\lambda X(t)-\psi\left(\lambda\right)r}
\mathrm{d}t
\right)
\nonumber\\
\hspace{-0.3cm}&=&\hspace{-0.3cm}
\int_{x}^{a}\mathrm{e}^{\lambda \xi(w)}
\exp\left(-\int_{x}^{w} \frac{\ell^{(q)\prime}_{r}(\overline{\xi}(z))}{\ell_{r}^{(q)}(\overline{\xi}(z))}\,\mathrm{d}z\right)
\left[\frac{\ell_{r}^{(q)\prime}(\overline{\xi}(w))}{\ell_{r}^{(q)}(\overline{\xi}(w))}\right.
\nonumber\\
\hspace{-0.3cm}&&\hspace{-0.3cm}
\left.
\times
\left(\frac{\mathrm{e}^{\lambda \overline{\xi}(w)}\left(1-\mathrm{e}^{-(\psi(\lambda)-q)r}\right)}{\psi(\lambda)-q}-
\mathrm{e}^{\lambda \overline{\xi}(w)}\int_{0}^{\overline{\xi}(w)}W^{(q)}(z)\mathrm{e}^{-\lambda z}\mathrm{d}z-\int_{0}^{r}
\mathrm{e}^{-\psi(\lambda) s}\,\ell_{s}^{(q)}(\overline{\xi}(w))\mathrm{d}s\right)
\right.
\nonumber\\
\hspace{-0.3cm}&&\hspace{-0.3cm}
\left.
-
\frac{\lambda \mathrm{e}^{\lambda \overline{\xi}(w)}\left(1-\mathrm{e}^{-(\psi(\lambda)-q)r}\right)}{\psi(\lambda)-q}+
\lambda \mathrm{e}^{\lambda \overline{\xi}(w)}\int_{0}^{\overline{\xi}(w)}W^{(q)}(z)\mathrm{e}^{-\lambda z}\mathrm{d}z
\right.
\nonumber\\
\hspace{-0.3cm}&&\hspace{-0.3cm}
\left.
+W^{(q)}(\overline{\xi}(w))+\int_{0}^{r}
\mathrm{e}^{-\psi(\lambda) s}\,\ell_{s}^{(q)\prime}(\overline{\xi}(w))\mathrm{d}s
\right]\,\mathrm{d}w
.\nonumber
\end{eqnarray}
\end{cor}

\begin{proof}
Letting $f(x):=\mathrm{e}^{\lambda x-\psi\left(\lambda\right)r}$ in Theorem \ref{thm3}, or using the compensation formula together with Remark \ref{rem1}, one can get the desired result.
\end{proof}

\vspace{0.3cm}
Recall the definition of $V_k^{\xi}\left( x;b \right)$ at the end of Section 2.
The following result generalizes (20) in Czarna and Palmowski (2014), and Propositions 1 and 2 in Renaud and Zhou (2007).

\begin{thm}\label{thm3.4}
For any $q\geq0$ and $k\geq1$, we have
 \begin{eqnarray} \label{momentsofD1}
&&{V_k^{\xi}}\left( x;b \right)
= \int_{b}^{\infty}kV_{k-1}(z)\exp \left( { - \int_x^z {\frac{{{\ell_{r}^{\left( kq \right)\prime}} \left( {\bar \xi \left( w \right)} \right)}}{{{\ell_{r}^{\left( {kq } \right)}}\left( {\bar \xi \left( w \right)} \right)}}} \mathrm{d}w} \right)\mathrm{d}z,\quad x \in \left( {-\infty,b} \right],
\nonumber
\end{eqnarray}
where
\begin{align}
V_k\left( x \right) = \int_{x}^{\infty}kV_{k-1}(z)\exp \left( { - \int_x^z {\frac{{{\ell_{r}^{\left( kq \right)\prime}} \left( {\bar \xi \left( w \right)} \right)}}{{{\ell_{r}^{\left( {kq } \right)}}\left( {\bar \xi \left( w \right)} \right)}}} \mathrm{d}w} \right)\mathrm{d}z,\quad x\in(-\infty,\infty),\nonumber
\end{align}
with $V_0\left( x \right)\equiv1$.
\end{thm}

\begin{proof}
For $\epsilon>0$ and integer $n\geq1$, we have
\begin{equation} \label{expression1oflemma2}
{\mathbb{E}_b}\left( {{{\left( {\int_0^{\tau _{b +  \epsilon }^ + } {{\mathrm{e}^{ - q s}}D\left( s \right)\mathrm{d}s} } \right)}^n}{{\mathbf{1}}_{\{ {\tau _{b +  \epsilon }^ +  < {\kappa^{\xi}_{r} }} \}}}} \right) = o\left( { \epsilon } \right),
\end{equation}
and
  \begin{equation} \label{expression2oflemma2}
{\mathbb{E}_b}\left( {{{\left( {\int_0^{{\kappa^{\xi}_{r} }} {{\mathrm{e}^{ - q s}}\mathrm{d}D\left( s \right)} } \right)}^n}{{\mathbf{1}}_{\{ {{\kappa^{\xi}_{r} } < \tau _{b +\epsilon }^ + } \}}}} \right) = o\left(\epsilon \right).
\end{equation}
Actually, $X\left( {\tau _{b + \epsilon }^ + } \right) = b + \epsilon $ implies $ D(s)\leq  \epsilon$ for all $s\in[0,{\tau _{b + \epsilon }^ + }]$. Hence, the left hand side of \eqref{expression1oflemma2} is less than
\begin{align}
 &\ { \epsilon^n }{\mathbb{E}_b}\left[ {{{\left( {\int_0^{\tau _{b +  \epsilon }^ + } {{\mathrm{e}^{ - q s}}\mathrm{d}s} } \right)}^n}{{\mathbf{1}}_{\left\{ {\tau _{b +  \epsilon }^ +  < {\kappa^{\xi}_{r} }} \right\}}}} \right]
 \nonumber \\
\leq &
\frac{\epsilon^{n}}{q^{n}}\left( {{\mathbb{E}_b}\left[ {{{\mathbf{1}}_{\left\{ {\tau _{b + \epsilon}^ +  < {\kappa^{\xi}_{r} }} \right\}}}} \right] - {\mathbb{E}_b}\left[ {{\mathrm{e}^{ - q \tau _{b + \epsilon}^ + }}{{\mathbf{1}}_{\left\{ {\tau _{b + \epsilon}^ +  < {\kappa^{\xi}_{r} }} \right\}}}} \right]} \right)
 \nonumber \\
=& \frac{\epsilon^{n}}{q^{n} }\left( {\exp \left( { - \int_b^{b +  \epsilon } {\frac{
{\ell_{r}^\prime \left( {\bar \xi \left( w \right)} \right)}}{{\ell_{r}\left( {\bar \xi \left( w \right)} \right)}}} \mathrm{d}w} \right) - \exp \left( { - \int_b^{b +  \epsilon } {\frac{{{\ell_{r}^{\left( q  \right)\prime}} \left( {\bar \xi \left( w \right)} \right)}}{{{\ell_{r}^{\left( q  \right)}}\left( {\bar \xi \left( w \right)} \right)}}} \mathrm{d}w} \right)} \right) = o\left( { \epsilon } \right).\nonumber
\end{align}
which gives (\ref{expression1oflemma2}). By integration by parts, the left hand side of \eqref{expression2oflemma2} can be rewritten as
\begin{align}
&\ {\mathbb{E}_b}\left[ {{{\left( {{\mathrm{e}^{ - q {\kappa^{\xi}_{r} }}}D( {{\kappa^{\xi}_{r} }} ) + q \int_0^{{\kappa^{\xi}_{r} }} {{\mathrm{e}^{ - q s}}D\left( s \right)\mathrm{d}s} } \right)}^{n}}{\mathbf{1}_{\{ {{\kappa^{\xi}_{r} } < \tau _{b + \epsilon }^ + } \}}}} \right]\
 \nonumber \\
  \le &\ {\mathbb{E}_b}\left[ {{{\left( {{\epsilon \mathrm{e}^{ - q {\kappa^{\xi}_{r} }}}  + \epsilon \int_0^{{\kappa^{\xi}_{r} }} { {q\mathrm{e}^{ - q s}} \mathrm{d}s} } \right)}^{n}}{\mathbf{1}_{\{ {{\kappa^{\xi}_{r} } < \tau _{b +  \epsilon }^ + } \}}}} \right]\
\nonumber \\
  =&\ {{\epsilon ^{n}}\left( {1 - \exp \left( { - \int_b^{b + \epsilon} {\frac{{\ell_{r}^\prime \left( {\bar \xi \left( w \right)} \right)}}{{\ell_{r}\left( {\bar \xi \left( w \right)} \right)}}} \mathrm{d}w} \right)} \right)}
 = o\left( { \epsilon } \right),\nonumber
\end{align}
which gives \eqref{expression2oflemma2}.
By \eqref{expression2oflemma2}, one has
\begin{align*}
{V_k}\left( b \right) = {\mathbb{E}_b}\left( {\left[D_{b}\right]^{k}{\mathbf{1}_{\{ {{\tau _{b +\epsilon }^{+}<\kappa^{\xi}_{r} } } \}}}} \right)   + o\left( { \epsilon } \right)
.
\end{align*}
Using the strong Markov property
and the Binomial Theorem,
one can rewrite ${\mathbb{E}_b}\left( {\left[D_{b}\right]^{k}{\mathbf{1}_{\{ \tau _{b +\epsilon }^{+}<\kappa^{\xi}_{r}\}}}} \right)$ as
\begin{align}
&\ {\mathbb{E}_b}\left[  {{{\sum\limits_{i = 0}^k {C_k^i\left( {\int_0^{\tau _{b +  \epsilon }^ + } {{\mathrm{e}^{ - q s}}\mathrm{d}D\left( s \right)} } \right)} }^i}{{\left( {\int_{\tau _{b +  \epsilon }^ + }^{{\kappa^{\xi}_{r} }} {{\mathrm{e}^{ - q s}}\mathrm{d}D\left( s \right)} } \right)}^{k - i}}{{\mathbf{1}}_{\{ {\tau _{b +  \epsilon }^ +  < {\kappa^{\xi}_{r} }} \}}}} \right]
 \nonumber \\
 =&\ {\mathbb{E}_b}\left[ {{{\sum\limits_{i = 0}^k {C_k^i\left( {\int_0^{\tau _{b +  \epsilon }^ + } {{\mathrm{e}^{ - q s}}\mathrm{d}D\left( s \right)} } \right)} }^i}{{\mathbf{1}}_{\left\{ {\tau _{b +  \epsilon }^ +  < {\kappa^{\xi}_{r} }} \right\}}}{{\left( {{\mathrm{e}^{ - q \tau _{b +  \epsilon }^ + }}} \right)}^{k - i}}{V_{k - i}}( b+\epsilon )} \right]
  \nonumber \\
 =&\ \sum\limits_{i = 0}^k {C_k^i} V_{k - i}( b+\epsilon )\mathbb{E}_b\left[ \left( \int_0^{\tau _{b +  \epsilon }^ + } \mathrm{e}^{ - q s}\mathrm{d}D\left( s \right) \right)^i \mathrm{e}^{ - (k - i)q \tau _{b +  \epsilon }^ +}  \mathbf{1}_{\{ {\tau _{b +  \epsilon }^ +  < {\kappa^{\xi}_{r} }} \}} \right]
  \nonumber \\
 =&\ \sum\limits_{i = 0}^k {C_k^i} V_{k - i}( b+\epsilon ){\mathbb{E}_b}\left[ {{\mathrm{e}^{ - (k - i)q \tau _{b +  \epsilon }^ +}}\sum\limits_{j = 0}^i {C_i^j
 \epsilon^{j} {\mathrm{e}^{ -j q \tau _{b +  \epsilon }^ + }}{{\left( {q \int_0^{\tau _{b +  \epsilon }^ + } {{\mathrm{e}^{ - q s}}D\left( s \right)\mathrm{d}s} } \right)}^{i - j}}} {{\mathbf{1}}_{\{ {\tau _{b + \epsilon }^ +  < {\kappa^{\xi}_{r} }} \}}}} \right]
 \nonumber \\
 =&\ \sum\limits_{i = 0}^k {C_k^i} V_{k - i}( b+\epsilon )\sum\limits_{j = 0}^i {C_i^j} {\epsilon ^j}{\mathbb{E}_b}\left[ {{\mathrm{e}^{ - q \left( {k - i + j} \right)\tau _{b + \epsilon}^ + }}{q ^{i - j}}{{\left( {\int_0^{\tau _{b + \epsilon}^ + } {{\mathrm{e}^{ - q s}}D\left( s \right)\mathrm{d}s} } \right)}^{i - j}}{{\mathbf{1}}_{\{ {\tau _{b + \epsilon}^ +  < {\kappa^{\xi}_{r} }} \}}}} \right],
\label{eqn:pass:-:3}
\end{align}
where the following identity
\begin{align}
\mathbb{E}_b\left(\left.\left(\int_{\tau _{b +  \epsilon }^ + }^{{\kappa^{\xi}_{r} }} {{\mathrm{e}^{ - q s}}\mathrm{d}D\left( s \right)}\right)^{i}\right|\mathcal{F}_{\tau _{b +  \epsilon }^ +}\right)
=&\mathbb{E}_b\left(\left.\left(\int_{\tau _{b +  \epsilon }^ + }^{{\kappa^{\xi}_{r} }} {{\mathrm{e}^{ - q s}}
\mathrm{d}\left( (\bar{X}(s)-(b+\epsilon))\vee 0 \right)}\right)^{i}\right|\mathcal{F}_{\tau _{b +  \epsilon }^ +}\right)
\nonumber \\
 =& \mathrm{e}^{ - iq \tau _{b +  \epsilon }^ + }\, V_{i}( b+\epsilon ),\quad i\geq0,\nonumber
\end{align}
is used for the first equation of \eqref{eqn:pass:-:3}.
Due to \eqref{expression1oflemma2} and \eqref{expression2oflemma2}, one can keep only those summands with $j=i=1$ or $j=i=0$ in (\ref{eqn:pass:-:3}), and then
\begin{align}
V_k\left( b \right)=&  {V_k}\left( b+\epsilon \right){\mathbb{E}_b}\left[ {{\mathrm{e}^{ - kq \tau _{b +  \epsilon }^ + }}{{\mathbf{1}}_{\left\{ {\tau _{b +  \epsilon }^ +  < {\kappa^{\xi}_{r} }} \right\}}}} \right] + k{V_{k - 1}}\left( b+\epsilon \right){\mathbb{E}_b}\left[ { \epsilon {\mathrm{e}^{ - kq \tau _{b +  \epsilon }^ + }}{{\mathbf{1}}_{\left\{ {\tau _{b +  \epsilon }^ +  < {\kappa^{\xi}_{r} }} \right\}}}} \right]+ o\left( { \epsilon } \right)
  \nonumber \\
 =& \left( {{V_k}\left( b+\epsilon \right) + k\epsilon{V_{k - 1}}\left( b+\epsilon \right)} \right)
 \exp \left( { - \int_b^{b +  \epsilon } {\frac{{{\ell_{r}^{\left( {kq } \right)\prime}} \left( {\bar \xi \left( w \right)} \right)}}{{{\ell_{r}^{\left( {kq } \right)}}\left( {\bar \xi \left( w \right)} \right)}}\mathrm{d}w} } \right) + o\left( { \epsilon } \right),\nonumber
\end{align}
which can be rearranged as
\begin{align}\label{eqn:pass:-:4}
0=&\left(\frac{V_k\left( b+\epsilon \right)-V_k\left( b \right)}{\epsilon}
+
 k V_{k - 1}\left( b+\epsilon \right)\right)
 \exp \left( { - \int_b^{b +  \epsilon } {\frac{{{\ell_{r}^{\left( {kq } \right)\prime}} \left( {\bar \xi \left( w \right)} \right)}}{{{\ell_{r}^{\left( {kq } \right)}}\left( {\bar \xi \left( w \right)} \right)}}\mathrm{d}w} } \right)
 \nonumber\\
 &+V_k\left( b\right)\frac{-1+\exp \left( { - \int_b^{b +  \epsilon } {\frac{{{\ell_{r}^{\left( {kq } \right)\prime}} \left( {\bar \xi \left( w \right)} \right)}}{{{\ell_{r}^{\left( {kq } \right)}}\left( {\bar \xi \left( w \right)} \right)}}\mathrm{d}w} } \right) }{\epsilon}
 + o( 1 ).
\end{align}
Letting $\epsilon\rightarrow0$ in \eqref{eqn:pass:-:4} we get
\begin{align}\label{diff.eq.}
0=V_k^{\prime}\left( b \right)
+
 k V_{k - 1}\left( b\right)
-V_k\left( b\right)
\frac{{{\ell_{r}^{\left( {kq } \right)\prime}} \left( {\bar \xi \left( b \right)} \right)}}{{{\ell_{r}^{\left( {kq } \right)}}\left( {\bar \xi \left( b \right)} \right)}}.
\end{align}
By the standard method of variation of constant, one can obtain the solution of \eqref{diff.eq.} with boundary condition ${V_k}\left(\infty\right) = 0$ and $V_{0}(x)=1$  as

\begin{align}
V_k\left( b \right) = \int_{b}^{\infty}kV_{k-1}(z)\exp \left( { - \int_b^z {\frac{{{\ell_{r}^{\left( kq \right)\prime}} \left( {\bar \xi \left( w \right)} \right)}}{{{\ell_{r}^{\left( {kq } \right)}}\left( {\bar \xi \left( w \right)} \right)}}} \mathrm{d}w} \right)\mathrm{d}z.\nonumber
\end{align}

For $x \in \left(-\infty, b\right]$ and $k\geq1$, we have
 \begin{align}
V_k^{\xi}\left(x;b \right)
 =&\mathbb{E}_{x}\left[\left(\int_{\tau _{b}^{+}}^{\kappa^{\xi}_{r}}\mathrm{e}^{-q s}\,\mathrm{d}\left(\bar{X}(s)-b\right)\right)^{k}\mathbf{1}_{\{\tau _{b}^{+}<\kappa^{\xi}_{r} \}}\right]
 \nonumber\\
 =&\mathbb{E}_{x}\left(
 \mathrm{e}^{-k q \tau _{b}^{+}}\mathbf{1}_{\{\tau _{b}^{+}<\kappa^{\xi}_{r} \}}\right)V_{k}(b)
  \nonumber\\
 =&\exp \left( { - \int_x^b {\frac{{{\ell_{r}^{\left( kq \right)\prime}} \left( {\bar \xi \left( w \right)} \right)}}{{{\ell_{r}^{\left( {kq } \right)}}\left( {\bar \xi \left( w \right)} \right)}}} \mathrm{d}w} \right)
 \int_{b}^{\infty}kV_{k-1}(z)\exp \left( { - \int_b^z {\frac{{{\ell_{r}^{\left( kq \right)\prime}} \left( {\bar \xi \left( w \right)} \right)}}{{{\ell_{r}^{\left( {kq } \right)}}\left( {\bar \xi \left( w \right)} \right)}}} \mathrm{d}w} \right)\mathrm{d}z.
  \nonumber
   \end{align}
The proof is complete.
\end{proof}

\begin{rem}
In the above proof, we borrow a Binomial argument from Renaud and Zhou (2007).
 But our Parisian draw-down time related arguments are more involved because we need to keep track of the running supreme process of $X$. In addition,  a differential equation argument was employed.
\end{rem}

\begin{rem}\label{cor3.6}

For $k=1$ we present an alternative more transparent argument.
Given $a>b $, by \eqref{newv.tw.si.} we have
\begin{eqnarray}\label{expected discounted dividend identity1''}
\hspace{-0.3cm}\hspace{-0.3cm}
V_{1}(b)
\hspace{-0.3cm}&=&\hspace{-0.3cm}\mathbb{E}_{b} \left(\int_{0}^{\infty}\mathbf{1}_{\{L(t)<L(\tau_{a}^{+}\wedge \kappa^{\xi}_{r})\}}\mathrm{e}^{-q t}\,\mathrm{d}L(t)\right)
\nonumber\\
\hspace{-0.3cm}&=&\hspace{-0.3cm}\int_{0}^{\infty}\mathbb{E}_{b}\left(\mathrm{e}^{-q L^{-1}(y)}\mathbf{1}_{\{y<L(\tau_{a}^{+})=a-b, \,y<L(\kappa^{\xi}_{r})\}}\right)\,\mathrm{d}y
\nonumber\\
\hspace{-0.3cm}&=&\hspace{-0.3cm}\int_{0}^{a-b}\mathbb{E}_{b}\left(\mathrm{e}^{-q L^{-1}(y)}\mathbf{1}_{\{L^{-1}(y)<\kappa^{\xi}_{r}\}}\right)\,\mathrm{d}y
\nonumber\\
\hspace{-0.3cm}&=&\hspace{-0.3cm}\int_{b}^{a}\mathbb{E}_{b}\left(\mathrm{e}^{-q \tau_{z}^{+}}\mathbf{1}_{\{\tau_{z}^{+}<\kappa^{\xi}_{r}\}}\right)\,\mathrm{d}z
\nonumber\\
\hspace{-0.3cm}&=&\hspace{-0.3cm}\int_{b}^{a}
\exp\left(-\int_{b}^{z} \frac{\ell^{(q)\prime}_{r}(\overline{\xi}(w))}{\ell_{r}^{(q)}(\overline{\xi}(w))}\,\mathrm{d}w\right)
\,\mathrm{d}z
.
\end{eqnarray}
Letting $a\rightarrow\infty$ in \eqref{expected discounted dividend identity1''}, we obtain the desired result.
\end{rem}

\vspace{0.3cm}
\section{Application of the draw-down Parisian ruin results to spectrally negative  L\'{e}vy process reflected at its past supreme}
\setcounter{section}{5}\setcounter{equation}{0}

\vspace{0.3cm}

 Recall the dividend process $D$  defined  in \eqref{Div.pro.}. Let the corresponding risk process with dividends deducted according to the barrier strategy with barrier level $b$, be defined as
$$Y(t):=X(t)-D(t), \quad t\geq0.$$
For fixed $b\in(0,\infty)$, if we choose the general draw-down function $\xi$ such that
$$\xi(z):=\xi_{b}(z)=\left(z- b\right)\vee0,\quad z\in(-\infty,\infty),$$
then we have
$$\kappa_{r}^{\xi}: =\inf\{t>r: t-g_{t}^{\xi}>r\}, \,\,\mbox{ where }\,\, g_{t}^{\xi}: =\sup\{0\leq s \leq t: Y(s)\geq0\},$$
i.e., $\kappa_{r}^{\xi}=\kappa_{r}^{\xi_{b}}$ degenerates to the classical Parisian ruin time for the risk process $Y$.
In addition, by the definition of $\tau_{\xi}$ we have
$$\tau_{\xi_{b}}:=\inf\{t\geq0: Y(t)\leq0\},$$
which is the ruin time for the risk process $Y$.

\vspace{0.3cm}
The following result gives the potential measure of $Y$ upon the up-crossing time of level $a$ or the Parisian ruin time of $Y$.

\begin{cor}
For $b\in(0,\infty)$, $\xi=\xi_{b}$, $q,\lambda\geq0$, $r>0$, $a\geq x$ and bounded differentiable function $f$, we have
\begin{eqnarray}\label{pot.mea..1'''}
\hspace{-0.3cm}&&\hspace{-0.3cm}
\int_{0}^{\infty}\mathrm{e}^{-q \left(t-r\right)}
\mathbb{E}_{x}\left(f\left(Y(t)\right); \,t<\kappa_{r}^{\xi_{b}}\wedge \tau_{a}^{+}\right)
\mathrm{d}t
\nonumber\\
\hspace{-0.3cm}&=&\hspace{-0.3cm}
\int_{x}^{a}
\exp\left(-\int_{x}^{w} \frac{\ell^{(q)\prime}_{r}(z\wedge b)}{\ell_{r}^{(q)}(z\wedge b)}\,\mathrm{d}z\right)
\left[\frac{\ell_{r}^{(q)\prime}(w\wedge b)}{\ell_{r}^{(q)}(w\wedge b)}
\left(\int_{0}^{r}\mathrm{e}^{q(r-s)}\mathbb{E}_{}\left(f(w\wedge b+X(s))\right)\mathrm{d}s
\right.\right.
\nonumber\\
\hspace{-0.3cm}&&\hspace{-0.3cm}
\left.-\int_{0}^{w\wedge b}W^{(q)}(w\wedge b-z)
\mathbb{E}_{}\left(f(z+X(r))\right)\mathrm{d}z
\right.
\nonumber\\
\hspace{-0.3cm}&&\hspace{-0.3cm}
\left.-\int_{0}^{r}\mathbb{E}_{}\left(f(X(r-s))\right)\ell_{s}^{(q)}(w\wedge b)\mathrm{d}s\right)
\nonumber\\
\hspace{-0.3cm}&&\hspace{-0.3cm}
-
\int_{0}^{r}\mathrm{e}^{q(r-s)}\mathbb{E}_{}\left(f^{\prime}(w\wedge b+X(s))\right)
\mathrm{d}s
+\int_{0}^{w\wedge b}W^{(q)\prime}(w\wedge b-z)\mathbb{E}_{}\left(f(z+X(r))\right)\mathrm{d}z
\nonumber\\
\hspace{-0.3cm}&&\hspace{-0.3cm}
\left.\left.+W^{(q)}(0+)\,\mathbb{E}_{}\left(f(w\wedge b+X(r))\right)
+\int_{0}^{r}\mathbb{E}_{}\left(f(X(r-s))\right)\ell_{s}^{(q)\prime}(w\wedge b)\mathrm{d}s\right)\right]
\,\mathrm{d}w.\nonumber
\end{eqnarray}
\end{cor}

\begin{proof}
By replacing $\xi$ and $f(x,y)$ respectively with $\xi_{b}$ and $f(x-(y-b)\vee0)$ in Theorem \ref{thm3}, one obtains the desired result.
\end{proof}

\vspace{0.3cm}
The following result gives the joint Laplace transform involving the Parisian ruin time of $Y$.

\begin{cor}\label{}
For $q,\lambda\in[0,\infty)$, $a \in(-\infty,\infty)$ and $x\in(-\infty,a)$, we have
\begin{eqnarray}\label{}
\hspace{-0.3cm}&&\hspace{-0.3cm}
\mathbb{E}_{x}\left(\mathrm{e}^{-q \left(\kappa_{r}^{\xi_{b}}-r\right)}\mathrm{e}^{\lambda Y(\kappa_{r}^{\xi_{b}})}\mathbf{1}_{\{\kappa_{r}^{\xi_{b}}<\tau_{a}^{+}\}}\right)
\nonumber\\
\hspace{-0.3cm}&=&\hspace{-0.3cm}
\mathrm{e}^{\psi\left(\lambda\right)r} \int_{x}^{a}
\exp\left(-\int_{x}^{w} \frac{\ell^{(q)\prime}_{r}(z\wedge b)}{\ell_{r}^{(q)}(z\wedge b)}\,\mathrm{d}z\right)
\left[
\frac{\ell_{r}^{(q)\prime}(w\wedge b)}{\ell_{r}^{(q)}(w\wedge b)}
\left(\mathrm{e}^{\lambda (w\wedge b)}-\left(\psi(\lambda)-q\right)\right.\right.
\nonumber\\
\hspace{-0.3cm}&&\hspace{-0.3cm}
\times
\left.\left.\left(\mathrm{e}^{\lambda (w\wedge b)}\int_{0}^{w\wedge b}W^{(q)}(z)\mathrm{e}^{-\lambda z}\mathrm{d}z+\int_{0}^{r}
\mathrm{e}^{-\psi(\lambda) s}\,\ell_{s}^{(q)}(w\wedge b)\mathrm{d}s\right)\right)\right.
\nonumber\\
\hspace{-0.3cm}&&\hspace{-0.3cm}
\left.-
\lambda \mathrm{e}^{\lambda (w\wedge b)}+\left(\psi(\lambda)-q\right)
\left(\lambda \mathrm{e}^{\lambda (w\wedge b)}\int_{0}^{w\wedge b}W^{(q)}(z)\mathrm{e}^{-\lambda z}\mathrm{d}z
\right.\right.
\nonumber\\
\hspace{-0.3cm}&&\hspace{-0.3cm}
\left.\left.
+W^{(q)}(w\wedge b)+\int_{0}^{r}
\mathrm{e}^{-\psi(\lambda) s}\,\ell_{s}^{(q)\prime}(w\wedge b)\mathrm{d}s\right)\right]\,\mathrm{d}w.\nonumber
\end{eqnarray}
\end{cor}

\begin{proof}
Letting $\xi(w)=\xi_{b}(w)$ and $\varphi(w)=\mathrm{e}^{-\lambda \xi_{b}(w)+\psi(\lambda)r}$ in Theorem \ref{thm2}, one arrives at the desired result.
\end{proof}

\vspace{0.3cm}
The following result on the $k$-th moment of the discounted total dividends paid according to the barrier strategy with barrier $b$ until the Parisian ruin time for $Y$ is a direct consequence of Theorem \ref{thm3.4}. In particular, the corresponding result with $k=1$ recovers identity (20) in  Czarna and Palmowski (2014).

\begin{cor}\label{cor4.2}
For $b\in(0,\infty)$, $\xi=\xi_{b}$, $q\geq0$ and $k\geq1$, we have
 \begin{eqnarray} \label{momentsofD1}
&&{V_k^{\xi_{b}}}\left( x;b \right)
= k!\,\frac{{\ell_{r}^{\left( {kq } \right)}}\left(x\right)}{{\ell_{r}^{\left( {kq } \right)}}\left(b\right)}\prod\limits_{i=1}^{k}\frac{{{\ell_{r}^{\left( {iq } \right)}}\left( b \right)}}{{{\ell_{r}^{\left( iq \right)\prime}} \left( b \right)}},\quad x \in \left( {-\infty,b} \right],
\nonumber
\end{eqnarray}
and
\begin{eqnarray}
V_k\left( x \right)
\hspace{-0.3cm}&=&\hspace{-0.3cm}
{\ell_{r}^{\left( {kq } \right)}}\left(x\right)
\left(
\int_{x}^{b}\frac{kV_{k-1}(z)}{{\ell_{r}^{\left( {kq } \right)}}\left(z\right)}\mathrm{d}z
+\frac{k!}{{\ell_{r}^{\left( kq \right)\prime}} \left(b\right)}\prod\limits_{i=1}^{k-1}\frac{{{\ell_{r}^{\left( {iq } \right)}}\left( b \right)}}{{{\ell_{r}^{\left( iq \right)\prime}} \left( b \right)}}\right)
,\quad x\in(-\infty,b],
\nonumber
\end{eqnarray}
with $V_0\left( x \right)\equiv1$.
In particular, for  $k=1$ we have
\begin{eqnarray}
V_{1}^{\xi_{b}}(x;b)
\hspace{-0.3cm}&=&\hspace{-0.3cm}
\frac{\ell_{r}^{(q)}(x)}{\ell^{(q)\prime}_{r}(b)},\quad x\in(-\infty,b].
\nonumber
\end{eqnarray}
\end{cor}

\begin{proof}
For $x\in(-\infty,b]$, we have
\begin{align}
V_k\left( x \right) =& \int_{x}^{b}kV_{k-1}(z)\exp \left( { - \int_x^z {\frac{{{\ell_{r}^{\left( kq \right)\prime}} \left( w \right)}}{{{\ell_{r}^{\left( {kq } \right)}}\left( w \right)}}} \mathrm{d}w} \right)\mathrm{d}z
\nonumber\\
&+\exp \left( { - \int_x^b {\frac{{{\ell_{r}^{\left( kq \right)\prime}} \left( w \right)}}{{{\ell_{r}^{\left( {kq } \right)}}\left( w \right)}}} \mathrm{d}w} \right)
\int_{b}^{\infty}kV_{k-1}(z)\exp \left( { - \int_b^z {\frac{{{\ell_{r}^{\left( kq \right)\prime}} \left( b \right)}}{{{\ell_{r}^{\left( {kq } \right)}}\left( b \right)}}} \mathrm{d}w} \right)\mathrm{d}z
\nonumber\\
=&{\ell_{r}^{\left( {kq } \right)}}\left(x\right)
\left(
\int_{x}^{b}\frac{kV_{k-1}(z)}{{\ell_{r}^{\left( {kq } \right)}}\left(z\right)}\mathrm{d}z
+\frac{k!}{{\ell_{r}^{\left( kq \right)\prime}} \left(b\right)}\prod\limits_{i=1}^{k-1}\frac{{{\ell_{r}^{\left( {iq } \right)}}\left( b \right)}}{{{\ell_{r}^{\left( iq \right)\prime}} \left( b \right)}}\right)
,\quad x\in(-\infty,b].\nonumber
\end{align}
The proof is thus complete.
\end{proof}

\vspace{0.3cm}
Put
$$D_{\tau_{\xi_{b}}}:=\int_{0}^{\tau_{\xi_{b}}}\mathrm{e}^{-qt}\mathrm{d}D(t),$$
represents the present value of the accumulated dividends paid until the time of ruin for $Y$. In addition, for each $k\geq1$, we also introduce the $k$-th moment of $D_{\tau_{\xi_{b}}}$ as
$$U_{k}(x;b):=\mathbb{E}_{x}\left(\left[D_{\tau_{\xi_{b}}}\right]^{k}\right).$$

The following result recovers Proposition 1 and Proposition 2 in Renaud and Zhou (2007).

\begin{cor}\label{cor4.3}
For $b\in(0,\infty)$, $q\geq0$ and $k\geq1$, we have
 \begin{eqnarray} \label{momentsofD1}
&&U_{k}(x;b)
= k!\,\frac{{W^{\left( {kq } \right)}}\left(x\right)}{{W^{\left( {kq } \right)}}\left(b\right)}\prod\limits_{i=1}^{k}\frac{{{W^{\left( {iq } \right)}}\left( b \right)}}{{{W^{\left( iq \right)\prime}} \left( b \right)}},\quad x \in \left( {-\infty,b} \right].
\nonumber
\end{eqnarray}
In particular, when $k=1$ we have
\begin{eqnarray}
U_{1}(x;b)
\hspace{-0.3cm}&=&\hspace{-0.3cm}
\frac{W^{(q)}(x)}{W^{(q)\prime}(b)},\quad x\in(-\infty,b].
\nonumber
\end{eqnarray}
\end{cor}

\begin{proof}
Note that $\tau_{\xi_{b}}=\lim\limits_{r\rightarrow0}\kappa_{r}^{\xi_{b}}$ with probability 1. Note from \eqref{newv.tw.si.} that
\begin{eqnarray}\label{1}
\mathbb{E}_{x}\left(\mathrm{e}^{-q \tau_{b}^{+}}\mathbf{1}_{\{\tau_{b}^{+}<\kappa_{r}^{\xi_{b}}\}}\right)
=\exp\left(-\int_{x}^{b} \frac{\ell^{(q)\prime}_{r}(w)}{\ell_{r}^{(q)}(w)}\,\mathrm{d}w\right)
= \frac{\ell^{(q)}_{r}(x)}{\ell_{r}^{(q)}(b)},\quad x\in(-\infty,b],
\end{eqnarray}
and
\begin{eqnarray}\label{1'}
\mathbb{E}_{x}\left(\mathrm{e}^{-q \tau_{a}^{+}}\mathbf{1}_{\{\tau_{a}^{+}<\kappa_{r}^{\xi_{b}}\}}\right)
\hspace{-0.3cm}&=&\hspace{-0.3cm}\exp\left(-\int_{x}^{b} \frac{\ell^{(q)\prime}_{r}(w)}{\ell_{r}^{(q)}(w)}\,\mathrm{d}w\right)
\exp\left(-\int_{b}^{a} \frac{\ell^{(q)\prime}_{r}(b)}{\ell_{r}^{(q)}(b)}\,\mathrm{d}w\right)
\nonumber\\
\hspace{-0.3cm}&=&\hspace{-0.3cm}
\frac{\ell^{(q)}_{r}(x)}{\ell_{r}^{(q)}(b)}
\,\mathrm{e}^{-\frac{\ell^{(q)\prime}_{r}(b)}{\ell_{r}^{(q)}(b)}(a-b)},\quad x\in(-\infty,b], a\in(b,\infty).
\end{eqnarray}
By \eqref{1} together with \eqref{t.s.0}, we have
\begin{eqnarray}\label{2}
\lim\limits_{r\rightarrow0}\frac{\ell^{(q)}_{r}(x)}{\ell_{r}^{(q)}(b)}
\hspace{-0.3cm}&=&\hspace{-0.3cm}
\lim\limits_{r\rightarrow0}\mathbb{E}_{x}\left(\mathrm{e}^{-q \tau_{b}^{+}}\mathbf{1}_{\{\tau_{b}^{+}<\kappa_{r}^{\xi_{b}}\}}\right)
\nonumber\\
\hspace{-0.3cm}&=&\hspace{-0.3cm}\mathbb{E}_{x}\left(\mathrm{e}^{-q \tau_{b}^{+}}\mathbf{1}_{\{\tau_{b}^{+}<\tau_{\xi_{b}}\}}\right)
=\mathbb{E}_{x}\left(\mathrm{e}^{-q \tau_{b}^{+}}\mathbf{1}_{\{\tau_{b}^{+}<\tau_{0}^{-}\}}\right)
= \frac{W^{(q)}(x)}{W^{(q)}(b)}, \quad x\in(-\infty,b].
\end{eqnarray}
By \eqref{dra.d.t.s.} we have
\begin{eqnarray}\label{3}
\mathbb{E}_{x}\left(\mathrm{e}^{-q \tau_{a}^{+}}\mathbf{1}_{\{\tau_{a}^{+}<\tau_{\xi_{b}}\}}\right)
\hspace{-0.3cm}&=&\hspace{-0.3cm}
\exp\left(-\int_{x}^{a} \frac{W^{(q)\prime}(\bar{\xi}_{b}(w))}{W^{(q)}(\bar{\xi}_{b}(w))}\,\mathrm{d}w\right)
\nonumber\\
\hspace{-0.3cm}&=&\hspace{-0.3cm}
\frac{W^{(q)}(x)}{W^{(q)}(b)}
\,\mathrm{e}^{-\frac{W^{(q)\prime}(b)}{W^{(q)}(b)}(a-b)},\quad x\in(-\infty,b], a\in(b,\infty).
\end{eqnarray}
Combining \eqref{1'}, \eqref{2} and \eqref{3}
we further arrive at
\begin{eqnarray}\label{2'}
\lim\limits_{r\rightarrow0}\frac{\ell^{(q)\prime}_{r}(b)}{\ell_{r}^{(q)}(b)}
=\frac{W^{(q)\prime}(b)}{W^{(q)}(b)}.
\end{eqnarray}
The desired results follow from a combination of \eqref{2}, \eqref{2'} and Corollary \ref{cor4.2}.
 \end{proof}

\vspace{0.3cm}
If we choose $\xi$ such that $\xi\equiv0$, then the draw-down Parisian ruin time $\kappa_{r}^{\xi}$ of $X$ degenerates to the Parisian ruin time $\kappa_{r}$ of $X$. The following result gives a generalized version of potential measure for the Process $X$ killed upon up-crossing level $a\,(\geq x)$ or the Parisian ruin time of $X$.

\begin{cor}\label{cor4.4}
For $\xi\equiv0$, $q,\lambda\geq0$, $r>0$, $a\geq x$ and bounded bivariate function $f(x,y)$ which is differentiable with respect to $x$, we have
\begin{eqnarray}\label{pot.mea..1''}
\hspace{-0.3cm}&&\hspace{-0.3cm}
\int_{0}^{\infty}\mathrm{e}^{-q \left(t-r\right)}
\mathbb{E}_{x}\left(f(X(t),\bar{X}(t)); \,t<\kappa_{r}\wedge \tau_{a}^{+}\right)
\mathrm{d}t
\nonumber\\
\hspace{-0.3cm}&=&\hspace{-0.3cm}
\int_{x}^{a}
\frac{\ell^{(q)}_{r}(x)}{\ell_{r}^{(q)}(w)}
\left[\frac{\ell_{r}^{(q)\prime}(w)}{\ell_{r}^{(q)}(w)}
\left(\int_{0}^{r}\mathrm{e}^{q(r-s)}\mathbb{E}_{}\left(f(w+X(s),w)\right)\mathrm{d}s
\right.\right.
\nonumber\\
\hspace{-0.3cm}&&\hspace{-0.3cm}
\left.-\int_{0}^{w}W^{(q)}(w-z)
\mathbb{E}_{}\left(f(z+X(r),w)\right)\mathrm{d}z
\right.
\nonumber\\
\hspace{-0.3cm}&&\hspace{-0.3cm}
\left.-\int_{0}^{r}\mathbb{E}_{}\left(f(X(r-s),w)\right)\ell_{s}^{(q)}(w)\mathrm{d}s\right)
\nonumber\\
\hspace{-0.3cm}&&\hspace{-0.3cm}
-
\int_{0}^{r}\mathrm{e}^{q(r-s)}\mathbb{E}_{}\left(\frac{\partial}{\partial x}f(w+X(s),w)\right)
\mathrm{d}s
+\int_{0}^{w}W^{(q)\prime}(w-z)\mathbb{E}_{}\left(f(z+X(r),w)\right)\mathrm{d}z
\nonumber\\
\hspace{-0.3cm}&&\hspace{-0.3cm}
\left.\left.+W^{(q)}(0+)\,\mathbb{E}_{}\left(f(w+X(r),w)\right)
+\int_{0}^{r}\mathbb{E}_{}\left(f(X(r-s),w)\right)\ell_{s}^{(q)\prime}(w)\mathrm{d}s\right)\right]
\,\mathrm{d}w.\nonumber
\end{eqnarray}
\end{cor}

\begin{proof}
It is a direct consequence of Theorem \ref{thm3}.
\end{proof}

The following result gives the solution for the $k$-th moment of the accumulated discounted dividend payout until the Parisian ruin time $\kappa_{r}$ for $X$.

\begin{cor}\label{cor4.5}
For $\xi\equiv0$, $q\geq0$ and $k\geq1$, we have,
 \begin{eqnarray} \label{momentsofD1}
&&{V_k^{\xi}}\left( x;b \right)
= {{\ell_{r}^{\left( kq \right)}} \left( {x} \right)}\int_{b}^{\infty}\frac{kV_{k-1}(z)}{{{\ell_{r}^{\left( {kq } \right)}}\left( {z} \right)}}\mathrm{d}z,\quad x \in \left( {-\infty,b} \right],
\nonumber
\end{eqnarray}
where
\begin{align}
V_k\left( x \right) ={{\ell_{r}^{\left( kq \right)}} \left( {x} \right)} \int_{x}^{\infty}\frac{kV_{k-1}(z)}{{{\ell_{r}^{\left( {kq } \right)}}\left( {z} \right)}}\mathrm{d}z,\quad x\in(-\infty,\infty),\nonumber
\end{align}
with $V_0\left( x \right)\equiv1$.
\end{cor}

\begin{proof}
The desired result is a direct application of Theorem \ref{thm3.4} by letting $\xi\equiv0$.
\end{proof}

\bigskip

{\bf Acknowledgement}
Wenyuan Wang thanks Concordia University where the work on this paper was  completed during his visit.

\end{document}